\documentclass[a4paper,12pt]{article}
\usepackage[margin=2cm]{geometry}

\usepackage[english]{babel}
\usepackage[T1]{fontenc}

\usepackage{graphicx}
\usepackage{subfig}
\usepackage{array}
\usepackage{color}
\usepackage{tikz,tkz-tab}
\usepackage{stmaryrd}
\usepackage{amsmath}
\usepackage{amssymb}
\usepackage{dsfont}
\usepackage{mathrsfs}

\usepackage{amsthm}
\usepackage{mathtools}

\newtheorem{thm}{Theorem}
\newtheorem{cor}[thm]{Corollary}
\newtheorem{prop}[thm]{Proposition}
\newtheorem{lem}[thm]{Lemma}

\theoremstyle{definition}
\newtheorem{defn}[thm]{Definition}

\theoremstyle{remark}
\newtheorem{rk}[thm]{Remark}
\newtheorem{ex}[thm]{Example}

\newcommand{\C}{\mathbb C}
\newcommand{\R}{\mathbb R}
\newcommand{\N}{\mathbb N}
\newcommand{\Z}{\mathbb Z}
\newcommand{\W}{\mathcal W}

\newcommand{\MT}{Theorem~\ref{T:MainTheorem}}
\newcommand{\superhyper}{hyper-hyperbolic}

\begin{document}

\title
{
Generic singularities of line fields on 2D manifolds}

\author{Ugo Boscain\thanks{U. Boscain is with CNRS, 
CMAP, \'Ecole Polytechnique, Universit\'e Paris-Saclay, Palaiseau \& Team GECO, INRIA Saclay,  France,
{\tt ugo.boscain@polytechnique.edu}}, Ludovic Sacchelli\thanks{L. Sacchelli is with  
CMAP, \'Ecole Polytechnique, Universit\'e Paris-Saclay, Palaiseau \& Team GECO, INRIA Saclay,  France,
{\tt ludovic.sacchelli@polytechnique.edu}},
        and Mario Sigalotti
\thanks{M. Sigalotti is with INRIA Saclay, Team GECO \&
CMAP, \'Ecole Polytechnique, Universit\'e Paris-Saclay, Palaiseau,   France,
 {\tt mario.sigalotti@inria.fr}}
 \thanks{The authors are grateful to Andrei Agrachev and Jean Petitot and  for useful and stimulating discussions. This research has been partially supported by the European Research Council, ERC StG 2009 ``GeCoMethods'', contract n. 239748, by the iCODE institute (research project of the Idex Paris-Saclay),  and by the ANR project SRGI ANR-15-CE40-0018. This research benefited also from the support of the ``FMJH Program Gaspard Monge in optimization and operation research'' and from the support to this program from EDF.}
 }

\maketitle 

\begin{abstract} 
Generic singularities of line fields have been studied for lines of principal curvature of embedded surfaces. In this paper we propose an approach to classify generic singularities of general line fields on 2D manifolds. The idea is to identify line fields as bisectors of pairs of vector fields on the manifold, with respect to a given conformal structure. The singularities correspond to the zeros of the vector fields and the genericity is considered with respect to a natural topology in the space of pairs of vector fields.
Line fields at generic singularities turn out to be topologically equivalent to the Lemon, Star and Monstar singularities that one finds at umbilical points.
\end{abstract}

\section{Introduction}
\label{Section:Introduction}

A line field on a 2-dimensional manifold is a smooth map that associates with every $q\in M$ a line ({\it i.e.}, a 1-dimensional subspace) in $T_qM$. 
This is the definition used in \cite{Hopf}, where Hopf extends the classical Poincar\'e-Hopf Theorem to the case of line fields. 
Line fields appear often in nature as for instance in fingerprints \cite{kass1987analyzing,penrose1979topology,wang2007fingerprint}, liquid crystals  \cite{chandrasekhar1992liquid,Dogic2015Orientational_order_of_motile_defects_in_active_nematics,prost1995physics} and  in the pinwheel structure of the visual cortex V1 of mammals  \cite{boscain2014hypoelliptic,BoscainRossi2012,cittisarti,hubel1962receptive,petitot2008neurogeometrie}. Contrarily to what happens for vector fields, where the topology of the manifold forces the vector fields to have zeros,  the topology of the manifold forces line fields  to have singularities ({\it i.e.}, points where a line field is not defined).

Singularities of line fields are visible in nature as shown in Figure~\ref{Fig:Nature}. Two types of singularities are usually observed, one of index $1/2$ and one of index $-1/2$, which have different names depending on the context. 

Following Thom \cite{thom1972stabilite}, one expects that only singularities that do not disappear for small perturbations of the system are  easily observed in nature (see also \cite{Arnold1992}). For this reason, it is important to study which  singularities  are {\em structurally stable}. 
To define what structurally stable means, one needs two ingredients. First one needs a topology on the space of line fields. Second one needs a notion of local equivalence between line fields. The difficulty in studying this problem comes from the fact that there is no natural topology on the set of line fields, since 
the set of singular points depends on the line field itself.

This problem was completely solved in the case of lines of principal curvature on surfaces, since these line fields are given by the embedding of a surface in $\R^3$ and the natural topology is the one given by the embedding. Three types of singularities, called \emph{Lemon}, \emph{Monstar} and \emph{Star} (see \cite{berry1977umbilic}), were identified by Darboux in \cite{Darboux}   (see Figure~\ref{I:LMS}). It was proven in \cite{sotomayor1982structurally} that Lemon, Monstar and Star are the structurally stable singularities of lines of principal curvature with respect to the Whitney $C^3$-topology of immersions of a surface in $\R^3$. 

The purpose of this paper is to study the structurally stable singularities of line fields in a more general context than the one of lines of principal curvature. 
The starting point of the paper is to give a definition of line field  (that we call {\em proto-line-field}) that has a natural associated topology. 
For us a proto-line-field on a Riemannian surface is a pair of vector fields $X$ and $Y$ on $M$. The corresponding line field associated with the proto-line-field is the line field bisecting $X$ and $Y$. The angle is computed using the Riemannian metric, actually a conformal would be sufficient.  The zeros of $X$ and $Y$ become singularities of the associated line field. In Proposition~\ref{P:GlobalDef}, we prove that any line field with singularities can be realized in this way.

With this definition we naturally associate a topology on line fields, that is, the Whitney topology on pairs of vector fields on $M$.
The main result of the paper is that generically a proto-line-field has only structurally stable singularities, which are Lemon, Monstar or Star singularities. Hence the structurally stable singularities for lines of principal curvature are the same as for general proto-line-fields. 

Notice that in nature it is not easy to distinguish between the Lemon and the Monstar singularity since they have the same index (see Section \ref{SubSection:Index_of_line_fields}) and they look quite similar.
This is why the observation of singularities of line fields in nature usually reports only two behaviors, characterized by the index of the singularity. 
One important issue for singularities of line fields (in particular for finger ridges) is their parameterization by a model with few parameters and capable to capture   high  curvature
patterns (\cite{ICC07_Wang}). 
Our definition of proto-line-fields could be useful for such applications, since it could be used to detect fine properties, such as the difference between Lemon and Monstar singularities.

The structure of the paper is the following. In Section \ref{Section:Basic_definitions} we give the definition of proto-line-field, of local structural stability and we state our main result (\MT). In Section \ref{Section:Basic_Properties}, we establish some basic properties of proto-line-fields and we prove that every line field (possibly with singularities) can be realized as a proto-line-field. Moreover, following Hopf we introduce the index of a proto-line-field and we show how to compute it starting from the indices of $X$ and $Y$. We also deduce that the index of a singularity of a generic proto-line-field is $1/2$ or $-1/2$.

The main technical part of the paper consists of Sections \ref{Section:Linear_Euclidean_case} and \ref{Section:Linearization}. In Section \ref{Section:Linear_Euclidean_case} we study the case of linear proto-line-fields in the Euclidean plane and we classify them into three categories corresponding to the three exhibited singularities.  In Section \ref{Section:Linearization} we study the general problem via a blow up and make use of the classification obtained in the linear case to prove the \MT.

In Section \ref{Section:The_role_of_the_metric} we study the role of the Riemannian metric on the identification between a proto-line-field and the corresponding line field. In particular we observe that  bifurcations between Lemon and Monstar singularities can occur by changing the metric.
Finally in Section \ref{SubSection:How_to_construct_a_Riemannian_metric} we show how to reduce the number of ingredients necessary to define proto-line-fields by constructing a Riemannian metric starting from two vector fields.

\begin{figure}[ht!]
\begin{center}

\begin{minipage}{0.5\textwidth}
\begin{center}
        \subfloat[Stopping point]{\includegraphics[width=0.45\textwidth]{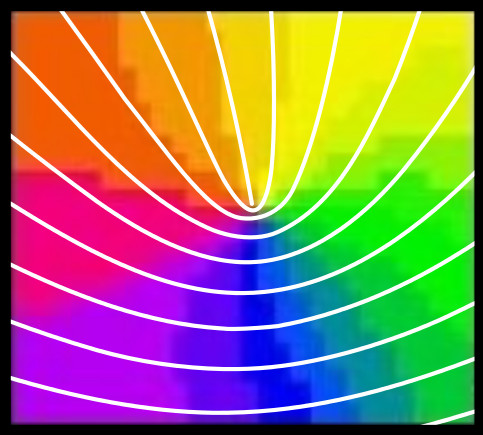}}\hspace{0.1\textwidth}\subfloat[Triple point]{\includegraphics[width=0.45\textwidth]{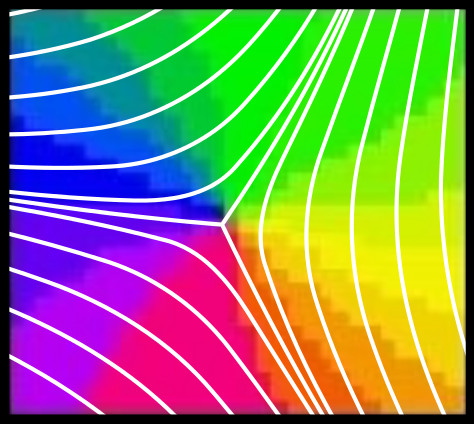}}  
        \end{center}
		\end{minipage}
		\hspace{0.03\textwidth}   
        \begin{minipage}{0.45\textwidth}
The pinwheel structure of the orientation columns of the visual cortex V1 can be modeled as a line field  whose singularities are the pinwheels. Clockwise pinwheels, also called stopping points, have index $1/2$ and counter-clockwise pinwheels, also called triple points, have index $-1/2$.
\end{minipage}

\vspace{\baselineskip}

\begin{minipage}{0.5\textwidth}
\begin{center}
        \subfloat[Core]{\includegraphics[width=0.45\textwidth]{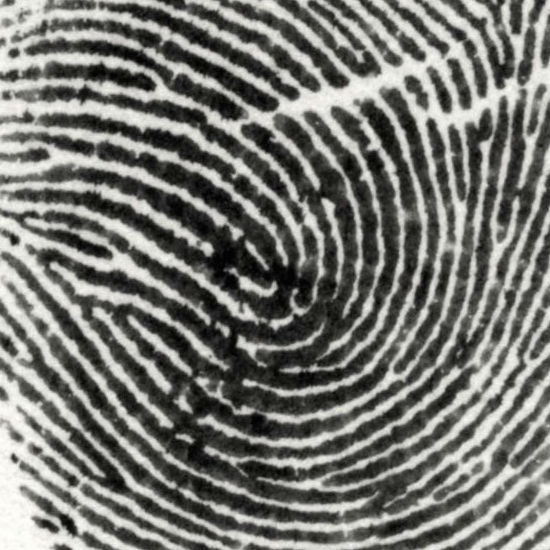}}\hspace{0.1\textwidth}\subfloat[Delta]{\includegraphics[width=0.45\textwidth]{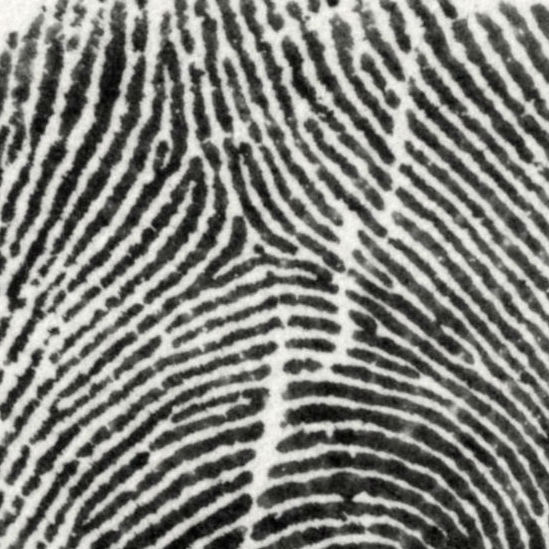}}   
        \end{center}
		\end{minipage}
		\hspace{0.03\textwidth}       
        \begin{minipage}{0.45\textwidth}
In an effort to classify fingerprints, the topology of the underlying line field in the ridge patterns can be used. Isolated singularities of index $1/2$ and $-1/2$ can be observed, and their total index is actually fixed by the number of fingers.
\end{minipage}

\vspace{\baselineskip}

\begin{minipage}{0.5\textwidth}
\begin{center}
        \subfloat[]{\includegraphics[width=0.45\textwidth]{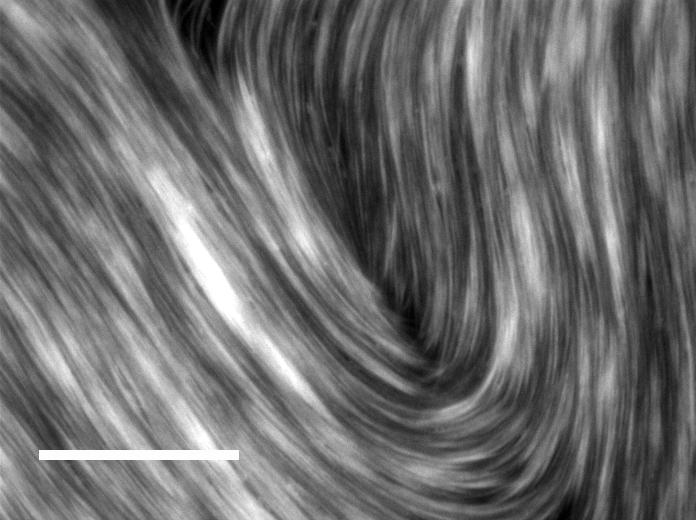}}\hspace{0.1\textwidth}\subfloat[]{\includegraphics[width=0.45\textwidth]{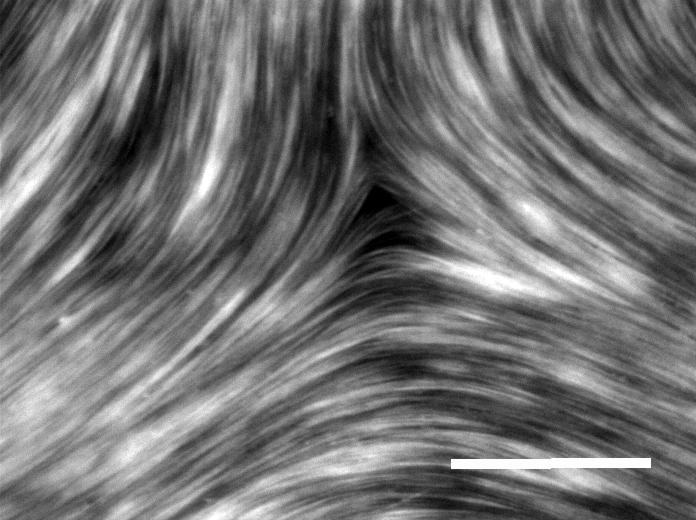}}
        \end{center}
		\end{minipage}
		\hspace{0.03\textwidth}   
        \begin{minipage}{0.45\textwidth}
        Singularities of index $\pm1/2$ can be observed in nematic liquid crystals. Perpendicularly to a 1-dimensional dislocation in the material, the liquid crystal can be modeled as a line field with singularities called disclinations. (Images kindly provided by Stephen J. DeCamp.)
\end{minipage}

\end{center}
\caption{Examples of singularities observed in nature with half-integer indices.}\label{Fig:Nature}
\end{figure}

\section{Basic definitions and statement of the main result}
\label{Section:Basic_definitions}

In this paper, manifolds and vector fields are assumed to be smooth, {\it i.e.}, $C^\infty$.

\begin{defn}
Let $(M,g)$ be a 2-dimensional Riemannian manifold. A \emph{proto-line-field}  is  a pair $(X,Y)$ of vector fields on $M$. Denote by $z_X$ and $z_Y$ the sets of zeros of $X$ and $Y$. The \emph{line field associated with $(X,Y)$}, denoted by  $B(X,Y)$, is the section of $PT(M\backslash (z_X\cup z_Y))$ defined at a point $p\in M\backslash (z_X\cup z_Y)$ as the line $B(X(p),Y(p))$  of $T_pM$ bisecting $(X(p),Y(p))$ for the metric $g(p)$.
\end{defn}

In the definition above, the metric $g$ is only used to measure angles. One could then replace $g$ by a conformal structure.
We are not assuming that $(M,g)$ is orientable. When angles are measured, it is implicitly meant that we are choosing a local orientation. In the following we will denote the angle measured with respect to the metric $g$ between the vectors $V$ and $W$ of $T_p M$ by $\angle_g[V,W]$. This angle should be understood modulo $2\pi$. We use the same notation to define the angle between two lines or between a vector and a line, in this case the angle should be understood modulo $\pi$.

\begin{defn}
A one-dimensional connected immersed submanifold $N$ of $M\backslash (z_X\cup z_Y)$ is said to be an \emph{integral manifold of the proto-line-field $(X,Y)$} if for any point $p$ of $N$, the tangent line to $N$ at $p$ is given by $B(X,Y)$. 
\end{defn}

By involutivity of one-dimensional distributions, $M\backslash (z_X\cup z_Y)$ can be foliated by integral manifolds of $(X,Y)$.

\begin{ex}\label{Ex:LMS}
We introduce here three proto-line-fields whose singularities correspond to the well-known Lemon, Monstar and Star singularities observed for lines of principal curvature. Their respective integral manifolds are represented in Figure \ref{I:LMS}.

The \emph{Lemon proto-line-field} is the pair of vector fields on $(\mathbb{R}^2,\mathrm{Eucl})$ defined by
$$
X_L
(x,y)
=
\begin{pmatrix}
x+y\\y-x
\end{pmatrix},\;
Y_L
(x,y)
=
\begin{pmatrix}
1\\1
\end{pmatrix}.
$$
The \emph{Monstar proto-line-field} is the pair of vector fields on $(\mathbb{R}^2,\mathrm{Eucl})$ defined by
$$
X_M
(x,y)
=
\begin{pmatrix}
x\\3y
\end{pmatrix},\;
Y_M
(x,y)
=
\begin{pmatrix}
1\\0
\end{pmatrix}.
$$
The \emph{Star proto-line-field} is the pair of vector fields on $(\mathbb{R}^2,\mathrm{Eucl})$ defined by
$$
X_S
(x,y)
=
\begin{pmatrix}
x\\-y
\end{pmatrix},\;
Y_S
(x,y)
=
\begin{pmatrix}
1\\0
\end{pmatrix}.
$$

\begin{figure}[ht!]
\begin{center}
        \subfloat[Lemon]{\includegraphics[width=0.28\textwidth]{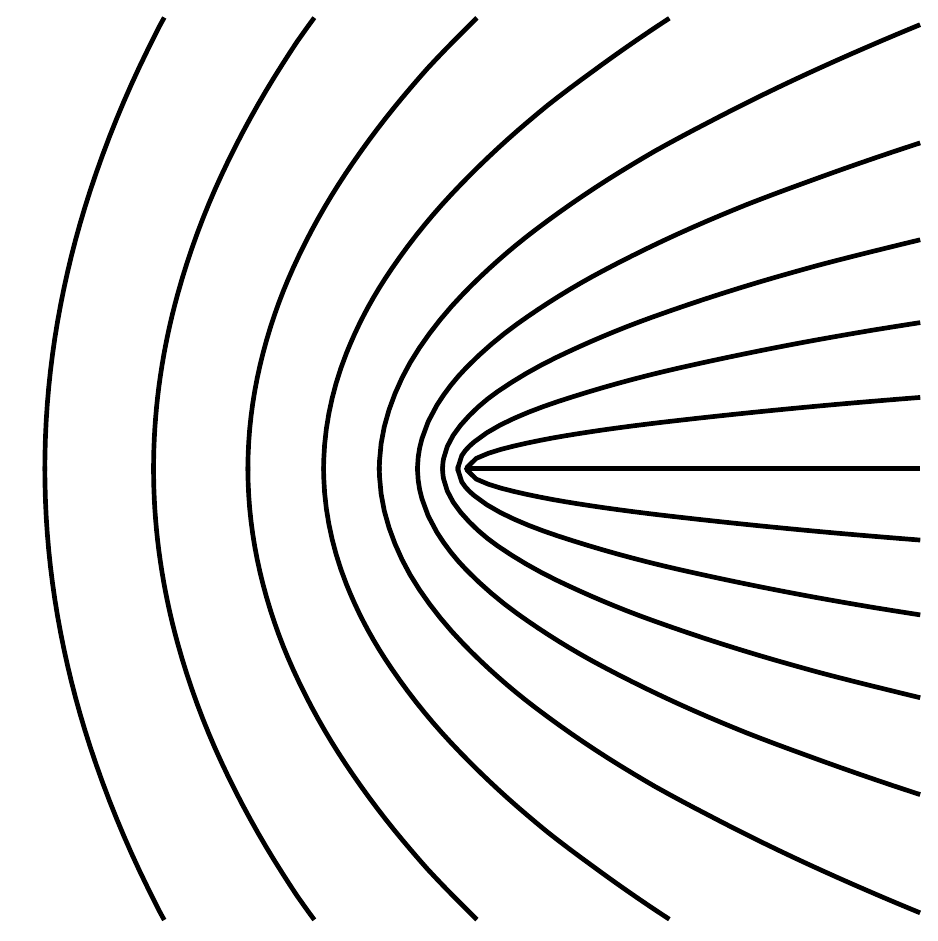}}
        \hspace{0.05\textwidth}
        \subfloat[Monstar]{\includegraphics[width=0.28\textwidth]{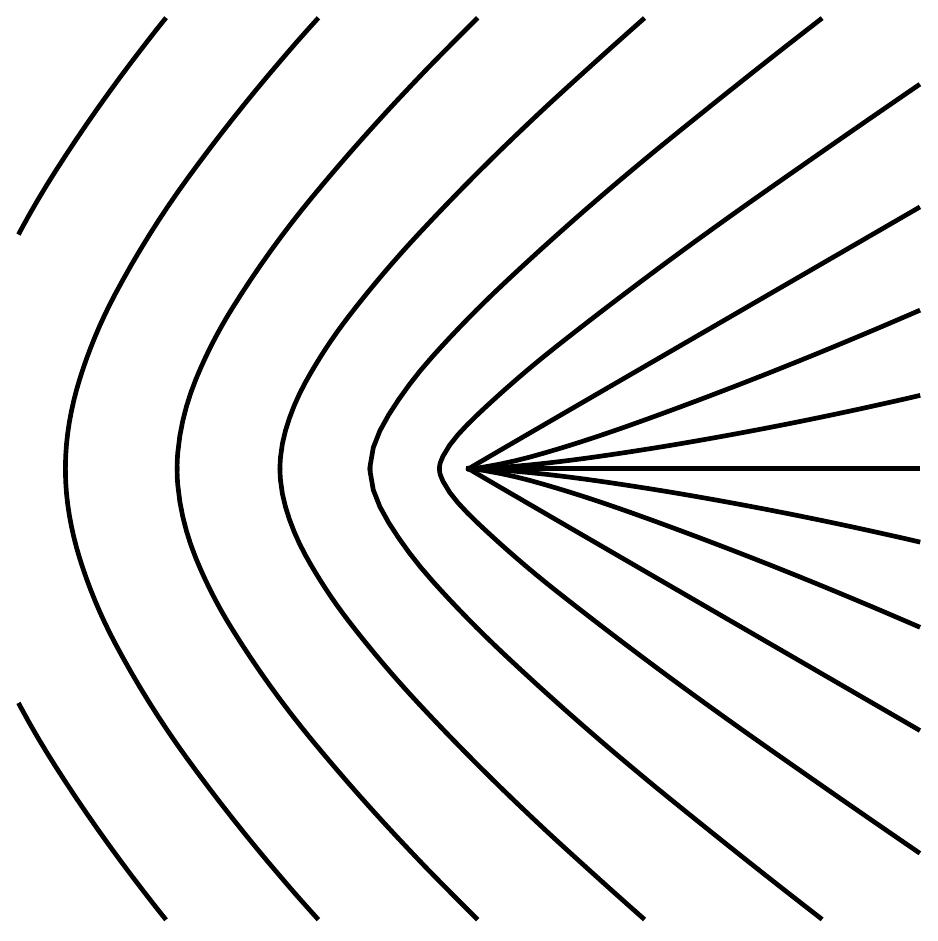}} 
        \hspace{0.05\textwidth}
        \subfloat[Star]{\includegraphics[width=0.28\textwidth]{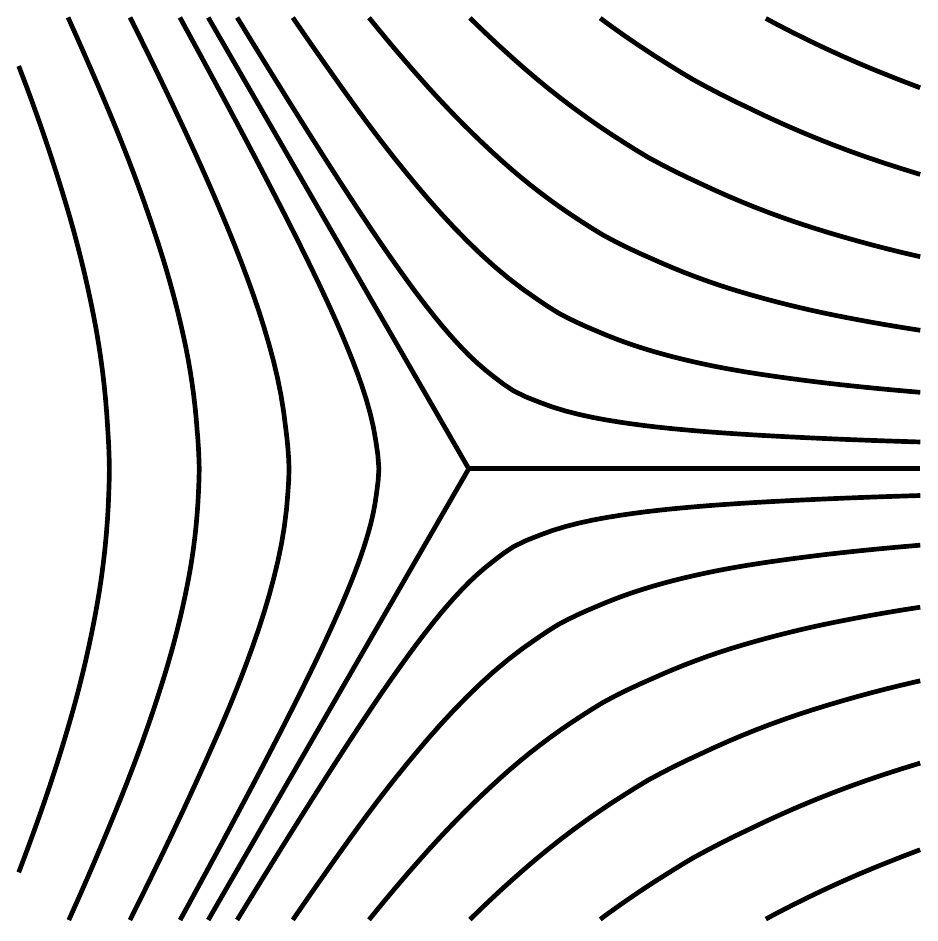}}
        \caption{Integral manifolds of the proto-line-fields of Example  \ref{Ex:LMS}.}
        \label{I:LMS}
\end{center}

\end{figure}

\end{ex}

For $k\in \N$, denote by $\W^k(M)$ the space of pairs of smooth vector fields of $M$ endowed with the product Whitney $C^k$-topology.

\begin{defn}
Let $(X,Y)$  be a proto-line-field on $(M,g)$, and $(X',Y')$ be a proto-line-field on $(M',g')$. Fix $p\in M$ and $p'\in M'$. Then $(X,Y)$ and $(X',Y')$ are said to be \emph{topologically equivalent at $p$ and $p'$} if there exist two neighborhoods $V_p$ and $W_{p'}$ of $p$ and $p'$ respectively and a homeomorphism $h:V_p\rightarrow W_{p'}$, with $h(p)=p'$, which takes the integral manifolds of $(X,Y)$ onto those of $(X',Y')$.
\end{defn}

\begin{defn}
Let $(X,Y)$ be a proto-line-field on $(M,g)$. We say that $(X,Y)$ has a Lemon (respectively, Monstar, Star) singularity at $p\in M$, if it is topologically equivalent to $(X_L,Y_L)$ (respectively, $(X_M,Y_M)$, $(X_S,Y_S)$) at $0$. We say that a singularity of a proto-line-field is \emph{Darbouxian} if it is either a Lemon, a Monstar or a Star. 
\end{defn}

\begin{defn}
A proto-line-field $(X,Y)$ on $M$ is said to be \emph{locally structurally stable at $p\in M$} if for any neighborhood $U_p\subset M$ of $p$ there exists a neighborhood $\mathcal{N}_{(X,Y)}$ of $(X,Y)$ with respect to $\W^1(M)$ such that for any $(X',Y')\in \mathcal{N}_{(X,Y)}$, $(X,Y)$ and $(X',Y')$ are topologically equivalent at $p$ and $q$, for some $q\in U_p$. Moreover, $(X,Y)$ is said to be \emph{locally structurally stable} if it is locally structurally stable at any point $p\in M$.
\end{defn}

Recall that a residual set in a topological space is a countable intersection of open and dense subsets.
We say that a property \emph{holds generically for a proto-line-field in $\W^1(M)$} if there exists a residual set $\mathcal{U}$ in $\W^1(M)$ such that the property is satisfied by every element of $\mathcal{U}$. In the case where $M$ is compact, we could actually replace {\it residual} by {\it open and dense} in the definition of genericity, and all results stated in this paper would still hold true.

\begin{thm}[Genericity theorem]\label{T:MainTheorem}
Generically with respect to $(X,Y)\in \W^1(M)$, the proto-line-field $(X,Y)$ is locally structurally stable and has only Darbouxian singularities.
\end{thm}

\section{Basic properties of proto-line-fields}
\label{Section:Basic_Properties}

\subsection{Every line field can be realized as a proto-line-field}
\label{SubSection:Every_line_field}

\begin{prop}\label{P:GlobalDef}
Let $(M,g)$ be a 2-dimensional Riemannian manifold, $K$ be a closed subset of $M$ and $L$ be a section of $PT(M\backslash K)$. There exist two vector fields $X$ and $Y$ such that $L=B(X,Y)$.
\end{prop}
\begin{proof}
Let us first fix the vector field $X$.
If $K$ is empty, by Poincar\'e-Hopf Theorem for line fields (see \cite{Hopf}), the Euler characteristic of $M$ is 0. Hence we can take as $X$  a never-vanishing smooth vector field on $M$.
In the case where $K$ is non-empty, we take instead as $X$ any vector field on $M$ vanishing at a single point $q$ belonging to $K$.

In the case in which $M$ is orientable, let $\alpha:M\backslash K \rightarrow \R/\pi\Z$ be the smooth function defined by $\angle[X,L]$ and $\varphi:M \rightarrow \R$ be a smooth function such that $\varphi(p)\neq 0$ for all $p\in M\backslash K$ and $\varphi$ is equal to 0 with all its derivatives on $K$. Define the smooth vector field $Y$ on $M$ by $Y(p)=\varphi(p)R_{2\alpha(p)}(X(p))$, where $R_{\theta}:TM\to TM$ denotes the fiber-wise rotation by an angle $\theta$.
By construction $L=B(X,Y)$. 

In the non-orientable case, even if $\alpha$ is not globally well defined on $M\backslash K$, the vector field $p\mapsto R_{2\alpha(p)}(X(p))$ is. Hence $Y$ can be defined as above and the same conclusion follows.
\end{proof}

\subsection{Index of line fields and hyperbolic singularities}
\label{SubSection:Index_of_line_fields}

Let $U$ be an open subset of $M$ and $p\in U$.
Following Hopf, we define the index of a section $L: U \setminus \{p\}\rightarrow PTM$ at $p$ as follows. Up to restricting $U$, we can assume it to be simply connected and we can consider a never vanishing vector field $Z$ on $U$.
Let $C:[0,1]\rightarrow U$ be a simple closed curve encircling $p$ counterclockwise. Then there exists a map $F:[0,1]\rightarrow TM$ such that $L(C(t))$ is the span of $F(t)$ for every $t\in [0,1]$.
Let $\angle\left[Z,F\right]_{C(t)}$ be the angle between $F(t)$ and $Z(C(t))$ with respect to the Riemannian metric $g$, and let $\delta_C \angle\left[Z,F\right]$ be the total signed variation of this angle on the interval $[0,1]$. We then define $j$ by
\[2\pi j = \delta_C \angle\left[Z,F\right].\]
Since $\angle\left[Z,F\right]_{C(0)}=\angle\left[Z,F\right]_{C(1)}\pmod{\pi}$, $2j$ is an integer and it can be shown that $j$ does not depend on $Z$, $C$ nor on $g$. We say that $j$ is the \emph{index of $L$ at $p$} and we write $\mathrm{ind}_p(L)=j$. We use the same symbol $\mathrm{ind}_p$ to denote the index of a vector field at the point $p$.

The following result holds (see \cite{Hopf}).
\begin{thm}[Poincar\'e-Hopf]
Let $(M,g)$ be a compact, orientable 2-dimensional Riemannian manifold of Euler characteristic $\chi(M)$, and let $L$ be a line field on $M$ with isolated singularities. Let $z_L$ be the set of singularities of $L$ and $j_p$ be the index of $L$ at $p$ for any $p\in z_L$. Then
$$
\sum_{p\in z_L}j_p=\chi(M).
$$
\end{thm}

\begin{ex}
Let us construct some examples of sections with arbitrary index using complex numbers. We assume that the indetermination of the logarithm is set on $\mathbb{R}_-$. Then we can define for any half-integer $j\in\frac{1}{2}\mathbb{Z}$ the smooth function 
\[
\begin{matrix}
\Phi_j: &\mathbb{C}\backslash\mathbb{R}_-&\longrightarrow &\mathbb{C}\\
	  &z&\longmapsto&z^j.
\end{matrix}
\]
The logarithm is not defined everywhere on $\mathbb{C}$, but the section defined by $L_j(z)=\arg(\Phi_j(z))\pmod{\pi}$ can be continuously extended on $\mathbb{C}\backslash \{0\}$ for any half-integer $j$. This section has a singularity of index $j$. Furthermore, notice that if $j$ is not an integer, then the section cannot be induced from a continuous vector field that vanishes only at $0$. (See Figure~\ref{Fig:ComplexLineFields}.)

\begin{figure}[!htb]
\begin{center}
         \subfloat[$j=-3/2$]{\includegraphics[width=0.36\textwidth]{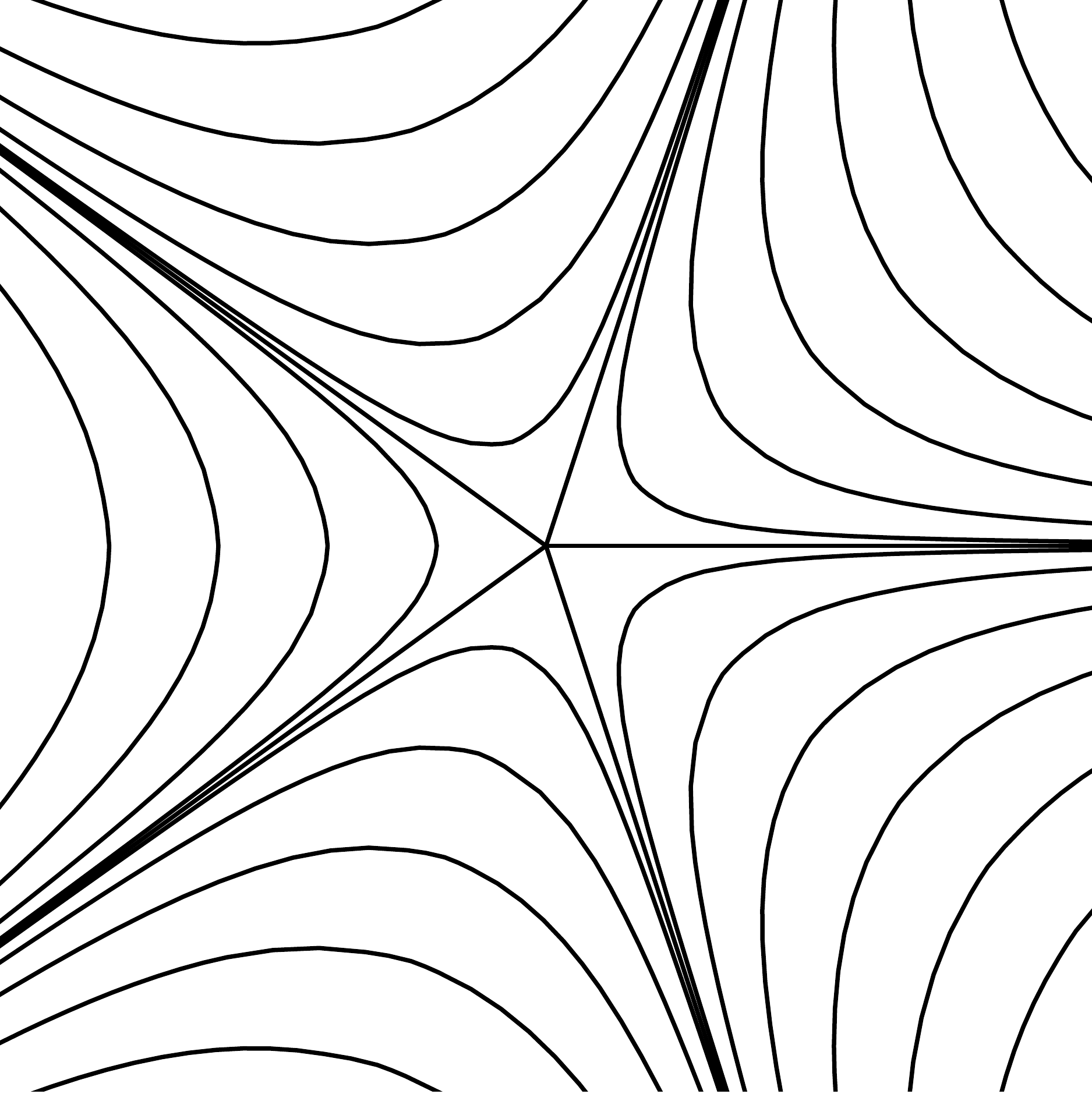}}
        \hspace{0.08\textwidth}
        \subfloat[$j=5/2$]{\includegraphics[width=0.36\textwidth]{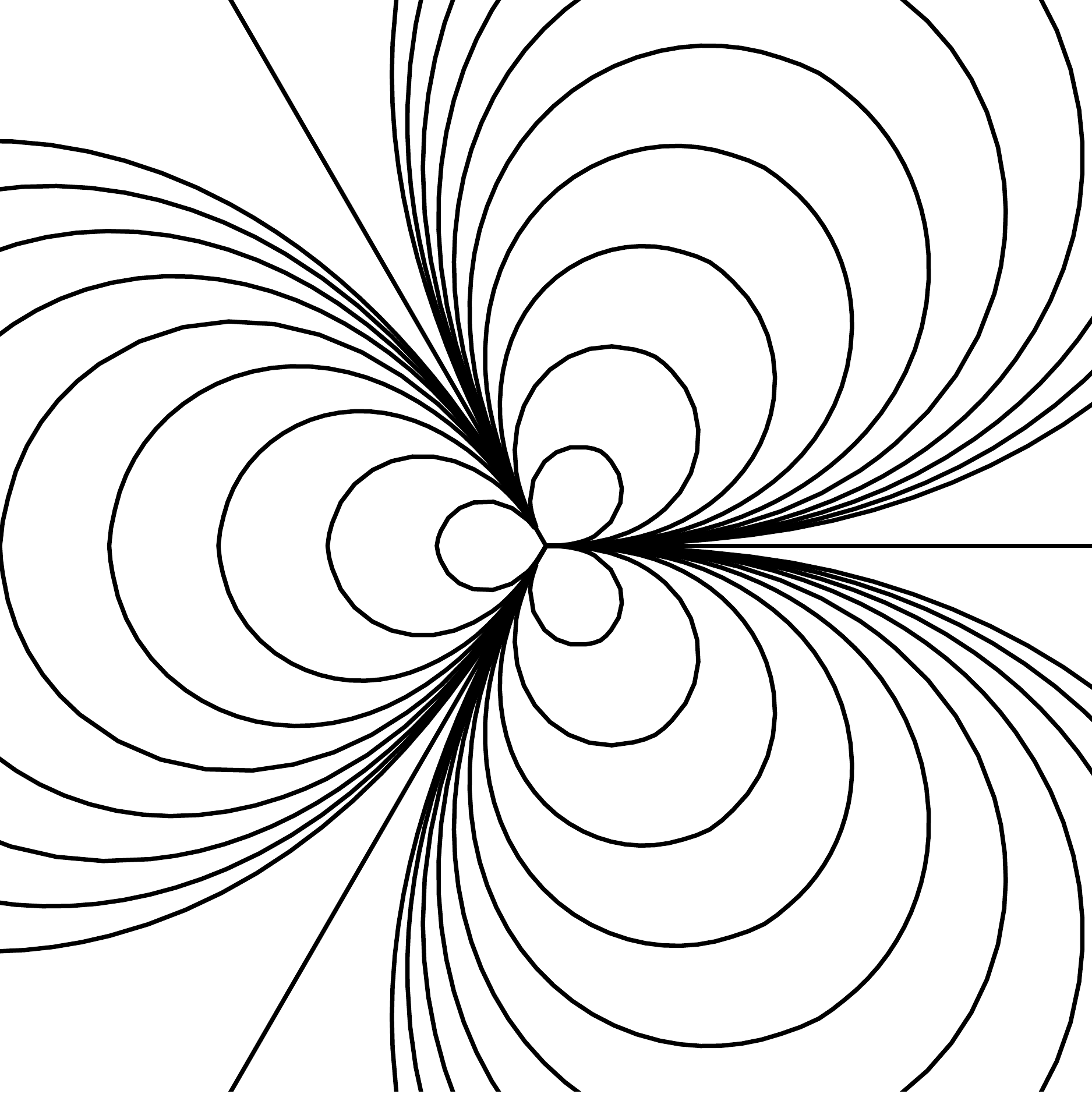}} 
\caption{Examples of integral manifolds of sections of $PT (\C \setminus \{0\})$ near half-integer index singularities.\label{Fig:ComplexLineFields}}
\end{center}
\end{figure}
\end{ex}

\begin{prop}\label{P:indices}
Let $(X,Y)$ be a proto-line-field on $(M,g)$. Given an isolated point $p$ of $z_X\cup z_Y$, we have $\mathrm{ind}_p(B(X,Y))=\frac{1}{2} \left(\mathrm{ind}_p(X)+\mathrm{ind}_p(Y)\right)$.
\end{prop}

\begin{proof}
Fix $p\in M$. Take $C$, $Z$ and $F:[0,1]\rightarrow TM$ as in the definition of the index of a singularity. Then for any $t\in [0,1]$
\begin{align*}
\angle\left[Z(C(t)),F(t)\right]
&=\angle\left[Z(C(t)),X(C(t))\right]+\angle\left[X(C(t)),F(t)\right] 
\\
&=\angle\left[Z(C(t)),X(C(t))\right]+\frac{1}{2} \angle\left[X(C(t)),Y(C(t))\right]
\\
&= \angle\left[Z(C(t)),X(C(t))\right]+\frac{1}{2}\left( \angle\left[Z(C(t)),Y(C(t))\right]+\angle\left[X(C(t)),Z(C(t))\right]\right)
\\
&=\frac{1}{2}\left( \angle\left[Z(C(t)),X(C(t))\right]+ \angle\left[Z(C(t)),Y(C(t))\right]\right),
\end{align*}
since $\angle\left[Z(C(t)),X(C(t))\right]=- \angle\left[X(C(t)),Z(C(t))\right]$.
Hence $$\mathrm{ind}_p(B(X,Y))=\frac{1}{2} \left(\mathrm{ind}_p(X)+\mathrm{ind}_p(Y)\right)$$ by definition of index.
\end{proof}

\begin{defn}
We say that a proto-line-field $(X,Y)$ has a \emph{hyperbolic singularity at a point $p\in M$} if one of the two vector fields has a hyperbolic singularity and the other is non-vanishing at $p$.
\end{defn}

\begin{prop}
A generic proto-line-field has only hyperbolic singularities. In particular its singularities have indices either $1/2$ or $-1/2$.
\end{prop}

\begin{proof}
As a straightforward consequence of Thom's transversality theorem (see, for instance, \cite[p. 82]{hirsch2012differential}), for a generic $(X,Y)\in \W^1(M)$, both $X$ and $Y$ have only hyperbolic singularities, and they do not vanish at the same point. This proves the first part of the statement, while the second follows from Proposition~\ref{P:indices}.
\end{proof}

\subsection{Example: Lines of principal curvature on a triaxial ellipsoid}
\label{SubSection:Example_Lines_of_curvature}

The study of lines of principal curvature on the triaxial ellipsoid is one of the most classical examples of this theory, that dates back to the work of Monge on the subject (see \cite{monge1796lignes,sotomayor2008lines}).

Consider the triaxial ellipsoid $\mathcal{E}$ of equation 
$$
\frac{x_1^2}{a}+\frac{x_2^2}{b}+\frac{x_3^2}{c}=1
$$ 
where we assume that $0<a< b< c$.
In order to introduce the coordinates on $\mathcal{E}$ used by Jacobi in \cite{jacobi1841ligne}, consider the  map from $\R/2\pi\Z\times \R/2\pi\Z$ onto $\mathcal{E}$ given by
\begin{align*} 
x_1 &=\sqrt{  \frac{a}{c - a}} \sin\varphi\sqrt{c\sin^2\psi + b\cos^2\psi - a},
\\
x_2 &= \sqrt{b} \cos\varphi \sin\psi, 
 \\
x_3 &=\sqrt{\frac{ c}{c-a}}  \cos\psi \sqrt{c - b\sin^2\varphi - a\cos^2\varphi}, 
\end{align*}
where $\varphi,\psi\in \R/2\pi\Z$. Although this map is a cover of $\mathcal{E}$ by the torus, the pair $(\varphi,\psi)$ is referred to as \emph{ellipsoidal coordinates}.
Their main interest for us is that level sets of $\varphi$ and $\psi$ ({\it i.e.}, curves on which either $\varphi$ or $\psi$ is constant) are the two sets of lines of curvature on the ellipsoid, and the umbilical points of the surface are situated at the points of coordinates $(\varphi,\psi)=(\pm\pi/2,\pm \pi)$.

The ellipsoidal coordinates are used by Jacobi to express the first integral of motion along geodesics on the ellipsoid. Indeed any geodesics on $\mathcal{E}$ can be described by an equation of the type
\begin{align*} 
\alpha 
&= \int \frac {\sqrt{b\sin^2\varphi + a\cos^2\varphi}} {\sqrt{c - b\sin^2\varphi - a\cos^2\varphi} \sqrt{(b-a)\cos^2\varphi - \beta}}\,\mathrm{d}\varphi
\\
 &\quad - \int \frac {\sqrt{c\sin^2\psi + b\cos^2\psi}} {\sqrt{c\sin^2\psi + b\cos^2\psi - a} \sqrt{(c-b)\sin^2\psi + \beta}} \,\mathrm{d}\psi
\end{align*}
where 
$$\beta = (b-a)\cos^2\varphi\sin^2\theta-(c-b)\sin^2\psi\cos^2\theta$$
is also constant along geodesics, with $\theta$ measuring the angle between the geodesic and 
the level set of $\psi$.

Among the geodesics, those which start at umbilical points of the ellipsoid satisfy very strong properties that we can use to characterize lines of principal curvature as integral manifolds of an explicitly identified proto-line-field.
First, umbilics are the only points on the ellipsoid for which the cut locus is reduced to a single point, the antipodal umbilic, and all geodesics between them have the same length. Therefore, by any non-umbilical point of the ellipsoid pass exactly two minimizing geodesics originating from the two pair of antipodal umbilics. On these geodesics, the constants of motion $\alpha$ and $\beta$ vanish.

Consider two non-antipodal umbilics $\Omega_1$ and $\Omega_2$ and two geodesics $t\mapsto(\varphi_1(t),\psi_1(t))$ and $t\mapsto(\varphi_2(t),\psi_2(t))$ starting at $\Omega_1$ and $\Omega_2$ respectively  and meeting at $(\varphi_1(t_1),\psi_1(t_1))=(\varphi_2(t_2),\psi_2(t_2))$. 
Since 
\begin{align*}
&(b-a)\cos^2\varphi_1(t_1)\sin^2\theta_1(t_1)-(c-b)\sin^2\psi_1(t_1)\cos^2\theta_1(t_1)=0,
\\
&(b-a)\cos^2\varphi_2(t_2)\sin^2\theta_2(t_2)-(c-b)\sin^2\psi_2(t_2)\cos^2\theta_2(t_2)=0,
\\
&\varphi_1(t_1)=\varphi_2(t_2),
\\
&\psi_1(t_1)=\psi_2(t_2),
\end{align*}
we have 
$$\sin^2\theta_1(t_1)=\sin^2\theta_2(t_2)$$
and thus 
$$
\theta_1(t_1)=\pm\theta_2(t_2)\pmod{\pi}.
$$

Since a geodesic is uniquely defined by its tangent line, $\theta_1=\theta_2\pmod{\pi}$ is excluded, and we have that $\theta_1(t_1)=-\theta_2(t_2)\pmod{\pi}$. 
By definition of $\theta_1$ and $\theta_2$, it follows that the level set of $\psi$ 
bisects the angle between the tangent lines of the two geodesics. We can write this fact in terms of proto-line-fields.

Consider the Riemannian exponential $\exp$ on $\mathcal{E}$ and the length $l$ of the geodesics connecting two umbilical points. Let
$ Y(x,y)=\sin\left(\frac{\pi}{l}\sqrt{x^2+y^2}\right)\begin{pmatrix}
x\\y
\end{pmatrix}$  be a vector field on the closed disc $\bar{B}_{\R^2}(0,l)$ of radius $l$, vanishing at $0$ and $\partial B_{\R^2}(0,l)$, and $
X_1=\exp_{\Omega_1}^* Y$, $X_2=\exp_{\Omega_2}^* Y$ be two vector fields on $\mathcal{E}$. By construction, for each $i\in \{1,2\}$, $X_i$ is tangent to the geodesics starting from $\Omega_i$ and vanishes at $\Omega_i$ and $-\Omega_i$. 

Then at the point $(\varphi_0,\psi_0)$, $X_i$ forms an angle $\theta_i\pmod{\pi}$, $i\in\{1,2\}$, with the level set $\{\psi=\psi_0\}$. Since $\theta_1=-\theta_2\pmod{\pi}$, the line bisecting $(X_1,X_2)$ for the metric on $\mathcal{E}$ induced from the Euclidean metric in $\R^3$ is either parallel or orthogonal to the level set $\{\psi=\psi_0\}$. In other words the line field bisecting the proto-line-field $(X_1,X_2)$ is one of the line fields of principal curvature, and the other is bisecting $(-X_1,X_2)$. (See Figure~\ref{Fig:Ellipsoid}.)

\begin{figure}[ht!]
\begin{center}
\includegraphics[width=.48\textwidth]{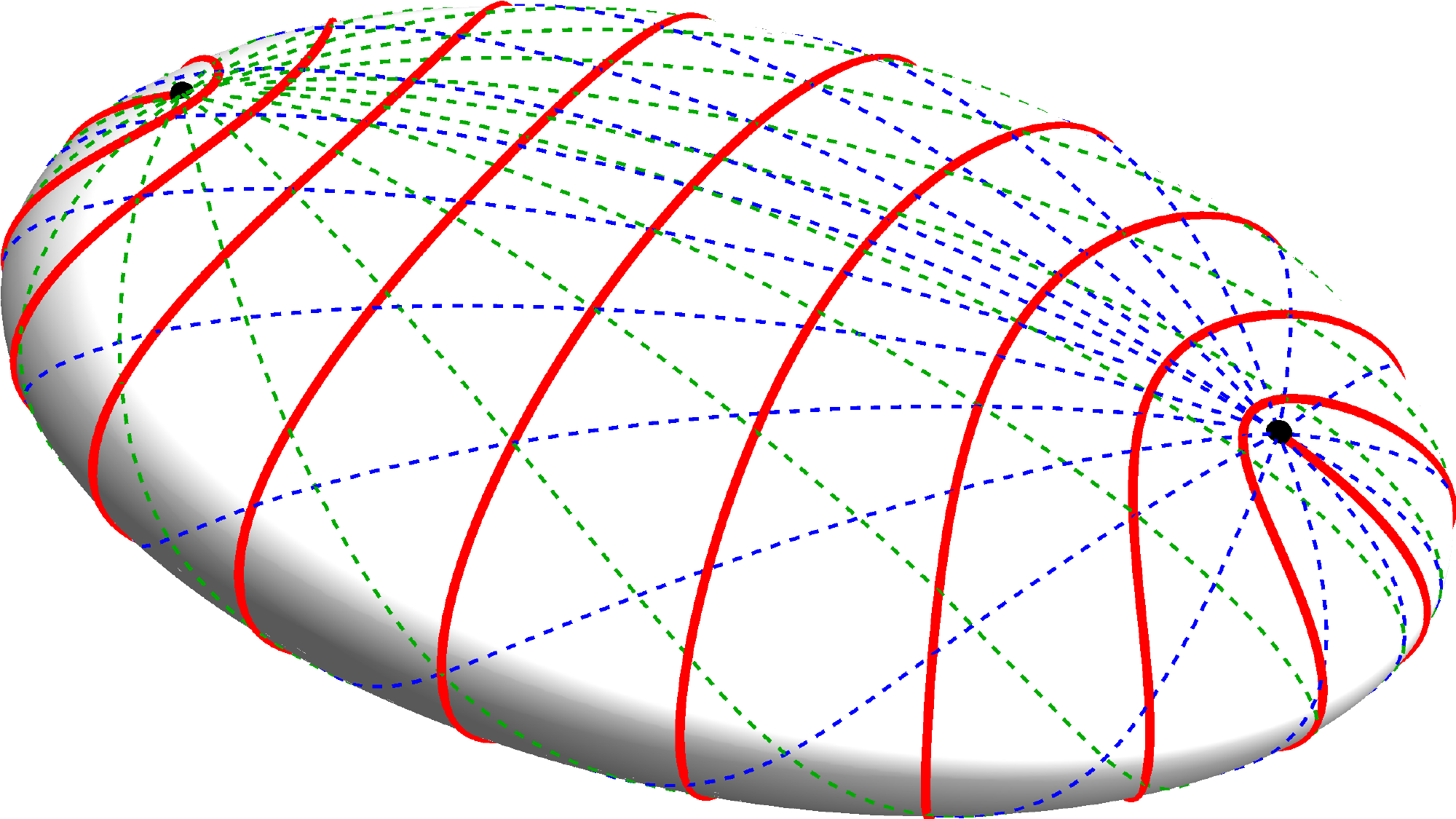}
\caption{Integral lines of $X_1$ and $X_2$ and integral manifolds of $(X_1,X_2)$, shown respectively in blue, green and red.\label{Fig:Ellipsoid}}
\end{center}

\end{figure}

\section{Linear Euclidean case}
\label{Section:Linear_Euclidean_case}

\begin{defn}
Let $(X,Y)$ be a proto-line-field on the Euclidean plane. We say that $(X,Y)$ is a \emph{linear proto-line-field} if one of the two vector fields is linear and the other one is constant. We say that $(X,Y)$ is \emph{linear hyperbolic} if it is linear and has a hyperbolic singularity at $0$.
\end{defn}

\begin{ex}
The three proto-line-fields presented in Example \ref{Ex:LMS} are linear hyperbolic.
\end{ex}

From now on, when considering a linear proto-line-field $(X,Y)$, we assume that $X$ is linear  and $Y$ is constant, that is, we consider $(X,Y)$ as an element of $\mathcal{M}_2(\mathbb{R})\times \mathbb{R}^2$, where $\mathcal{M}_2(\mathbb{R})$ denotes the space of square 2-by-2 matrices.

Consider a linear proto-line-field $L=(X,Y)$.  Along the rays $\{(r \cos\theta, r\sin \theta)\mid r>0\}$,  $\theta\in \R$, the direction of $X$ and $Y$ and thus of $B(X,Y)$ is constant. Hence we can define 
$\phi_L:\R/2\pi\Z \rightarrow \R/\pi\Z$ a parametrization of the direction of $B(X,Y)$ (after fixing an orthonormal basis basis) with 
$$
\phi_{L}(\theta)=
\angle_{\mathrm{Eucl}}
\left[
	\begin{pmatrix}
	1 \\ 0
	\end{pmatrix},B(L(\cos\theta,\sin\theta))
\right].
$$
We call \emph{fixed point of $L$} any point $\theta_0\in \R/2\pi\Z$ such that $\phi_L(\theta_0)=\theta_0 \pmod{\pi}$. A fixed point $\theta_0$ is said to be \emph{attractive} if $\frac{d\phi_L}{d\theta}(\theta_0)>1$, and \emph{repulsive} if $\frac{d\phi_L}{d\theta}(\theta_0)<1$.

\begin{thm}\label{T:thmcas}
Let $X$ be a linear hyperbolic vector field on $\mathbb{R}^2$. Then there exists a set $\mathcal{E}_X \subset \mathbb{R}^2$ made of finitely many lines through the origin such that if $Y\in \mathbb{R}^2\backslash \mathcal{E}_X$ then the linear proto-line-field
$L=(X,Y)$ satisfies one of the following properties
\begin{enumerate}
\item \label{PT:1}$\phi_L$ has a unique fixed point, which is repulsive;
\item \label{PT:2}$\phi_L$ has three fixed points, all in the same half-plane. We can then identify two external fixed points, which are repulsive, and one internal, which is attractive;
\item \label{PT:3}$\phi_L$ has three repulsive fixed points, which are not contained in a single half-plane.
\end{enumerate}

Moreover, the set of linear hyperbolic proto-line-fields which do not fall in one of the stated cases is given by 
\[
\left\{ L \mbox{ linear hyperbolic}\mid \exists\theta_0\in \R/2\pi\Z \mbox{ such that }
\phi_L(\theta_0)=\theta_0 \pmod{\pi},\;
\frac{d\phi_L}{d\theta}(\theta_0)=1
\right\}
\]
and has codimension 1 in $\mathcal{M}_2(\mathbb{R})\times \mathbb{R}^2$.
\end{thm}

\begin{figure}[ht]
\begin{center}
\begin{minipage}{.45\linewidth}
\begin{center}
		\subfloat[Case \ref{PT:1}: $\theta_1$ is repulsive.]{
\begingroup
  \makeatletter%
  \providecommand\color[2][]{%
    \errmessage{(Inkscape) Color is used for the text in Inkscape, but the package 'color.sty' is not loaded}%
    \renewcommand\color[2][]{}%
  }%
  \providecommand\transparent[1]{%
    \errmessage{(Inkscape) Transparency is used (non-zero) for the text in Inkscape, but the package 'transparent.sty' is not loaded}%
    \renewcommand\transparent[1]{}%
  }%
  \providecommand\rotatebox[2]{#2}%
  \ifx\svgwidth\undefined%
    \setlength{\unitlength}{0.1\textwidth}%
    \ifx\svgscale\undefined%
      \relax%
    \else%
      \setlength{\unitlength}{\unitlength * \real{\svgscale}}%
    \fi%
  \else%
    \setlength{\unitlength}{\svgwidth}%
  \fi%
  \global\let\svgwidth\undefined%
  \global\let\svgscale\undefined%
  \makeatother%
  \begin{picture}(10.4,5.3)%
    \put(0.2,0.3){\includegraphics[width=\textwidth]{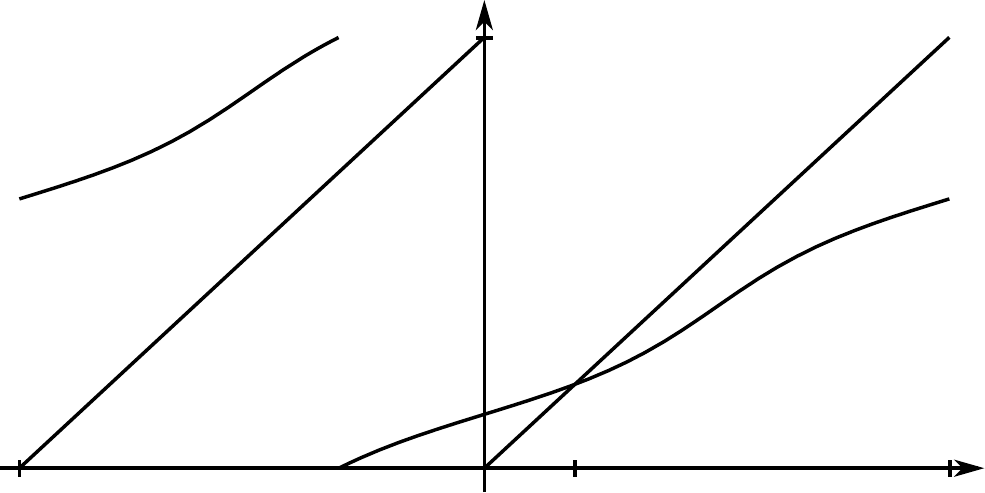}}%
    \put(0.05,0.0){\color[rgb]{0,0,0}\makebox(0,0)[lb]{$-\pi$\phantom{${}_1$}}}%
    \put(5.9,-0.1){\color[rgb]{0,0,0}\makebox(0,0)[lb]{$\theta_1$}}%
    \put(4.8,0){\color[rgb]{0,0,0}\makebox(0,0)[lb]{$0$}}%
    \put(9.75,0.0){\color[rgb]{0,0,0}\makebox(0,0)[lb]{$\pi$\phantom{${}_1$}}}%
    \put(5.25,4.8){\color[rgb]{0,0,0}\makebox(0,0)[lb]{$\pi$}}%
  \end{picture}%
\endgroup%
}
\end{center}
\end{minipage}
\hspace{0.05\textwidth}
\begin{minipage}{0.45\textwidth}
\begin{center}
\subfloat[Case \ref{PT:2}: $\theta_1$ and $\theta_3$ are repulsive, $\theta_2$ is attractive.]{
\begingroup%
  \makeatletter%
  \providecommand\color[2][]{%
    \errmessage{(Inkscape) Color is used for the text in Inkscape, but the package 'color.sty' is not loaded}%
    \renewcommand\color[2][]{}%
  }%
  \providecommand\transparent[1]{%
    \errmessage{(Inkscape) Transparency is used (non-zero) for the text in Inkscape, but the package 'transparent.sty' is not loaded}%
    \renewcommand\transparent[1]{}%
  }%
  \providecommand\rotatebox[2]{#2}%
  \ifx\svgwidth\undefined%
    \setlength{\unitlength}{0.1\textwidth}%
    \ifx\svgscale\undefined%
      \relax%
    \else%
      \setlength{\unitlength}{\unitlength * \real{\svgscale}}%
    \fi%
  \else%
    \setlength{\unitlength}{\svgwidth}%
  \fi%
  \global\let\svgwidth\undefined%
  \global\let\svgscale\undefined%
  \makeatother%
  \begin{picture}(10.4,5.3)%
    \put(0.2,0.3){\includegraphics[width=\textwidth]{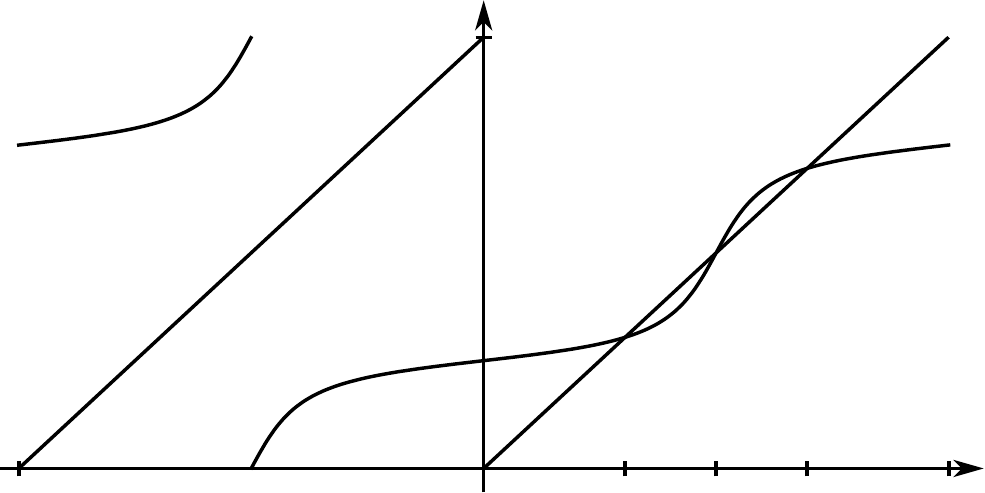}}%
    \put(0.05,0.0){\color[rgb]{0,0,0}\makebox(0,0)[lb]{$-\pi$}}%
    \put(9.75,0.0){\color[rgb]{0,0,0}\makebox(0,0)[lb]{$\pi$}}%
    \put(6.45,-0.1){\color[rgb]{0,0,0}\makebox(0,0)[lb]{$\theta_1$}}%
    \put(7.35,-0.1){\color[rgb]{0,0,0}\makebox(0,0)[lb]{$\theta_2$}}%
    \put(8.3,-0.1){\color[rgb]{0,0,0}\makebox(0,0)[lb]{$\theta_3$}}%
    \put(5.25,4.8){\color[rgb]{0,0,0}\makebox(0,0)[lb]{$\pi$}}%
    \put(4.8,0){\color[rgb]{0,0,0}\makebox(0,0)[lb]{$0$}}%
  \end{picture}%
\endgroup%
}
\end{center}
\end{minipage}

\vspace{0.5\baselineskip}

\begin{minipage}{0.45\textwidth}
\begin{center}
\subfloat[Case \ref{PT:3}: $\theta_1$, $\theta_2$ and $\theta_3$ are repulsive.]{         
\begingroup%
  \makeatletter%
  \providecommand\color[2][]{%
    \errmessage{(Inkscape) Color is used for the text in Inkscape, but the package 'color.sty' is not loaded}%
    \renewcommand\color[2][]{}%
  }%
  \providecommand\transparent[1]{%
    \errmessage{(Inkscape) Transparency is used (non-zero) for the text in Inkscape, but the package 'transparent.sty' is not loaded}%
    \renewcommand\transparent[1]{}%
  }%
  \providecommand\rotatebox[2]{#2}%
  \ifx\svgwidth\undefined%
    \setlength{\unitlength}{0.1\textwidth}%
    \ifx\svgscale\undefined%
      \relax%
    \else%
      \setlength{\unitlength}{\unitlength * \real{\svgscale}}%
    \fi%
  \else%
    \setlength{\unitlength}{\svgwidth}%
  \fi%
  \global\let\svgwidth\undefined%
  \global\let\svgscale\undefined%
  \makeatother%
  \begin{picture}(10.4,5.3)%
    \put(0.2,0.3){\includegraphics[width=\textwidth]{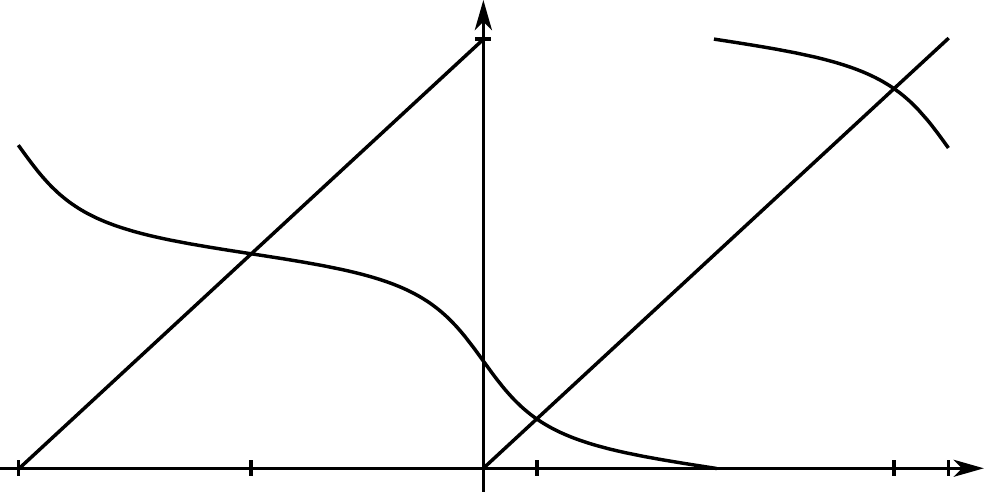}}%
    \put(0.05,0.0){\color[rgb]{0,0,0}\makebox(0,0)[lb]{$-\pi$}}%
    \put(9.75,0.0){\color[rgb]{0,0,0}\makebox(0,0)[lb]{$\pi$}}%
    \put(2.65,-0.1){\color[rgb]{0,0,0}\makebox(0,0)[lb]{$\theta_1$}}%
    \put(5.55,-0.1){\color[rgb]{0,0,0}\makebox(0,0)[lb]{$\theta_2$}}%
    \put(9.2,-0.1){\color[rgb]{0,0,0}\makebox(0,0)[lb]{$\theta_3$}}%
    \put(5.25,4.8){\color[rgb]{0,0,0}\makebox(0,0)[lb]{$\pi$}}%
    \put(4.8,0){\color[rgb]{0,0,0}\makebox(0,0)[lb]{$0$}}%
  \end{picture}%
\endgroup%
}
   \end{center}
\end{minipage}     

\caption[Examples of $\phi_L$ in each of the three cases of Theorem \ref{T:thmcas}.
These specific examples correspond to , , respectively.]{Examples of $\phi_L$ in each of the three cases of Theorem \ref{T:thmcas}.
These specific examples correspond to $L=\left(\begin{pmatrix}
3x\\2y
\end{pmatrix}, 
\begin{pmatrix}
1\\1
\end{pmatrix}\right)$, $L=\left(\begin{pmatrix}
4x\\y
\end{pmatrix}, 
\begin{pmatrix}
0\\1
\end{pmatrix}\right)$, and $L=\left(\begin{pmatrix}
x\\-3y
\end{pmatrix}, 
\begin{pmatrix}
0\\1
\end{pmatrix}\right)$ respectively.
\label{Fig:eachCaseOfThmCas}
}
\end{center}
\end{figure}
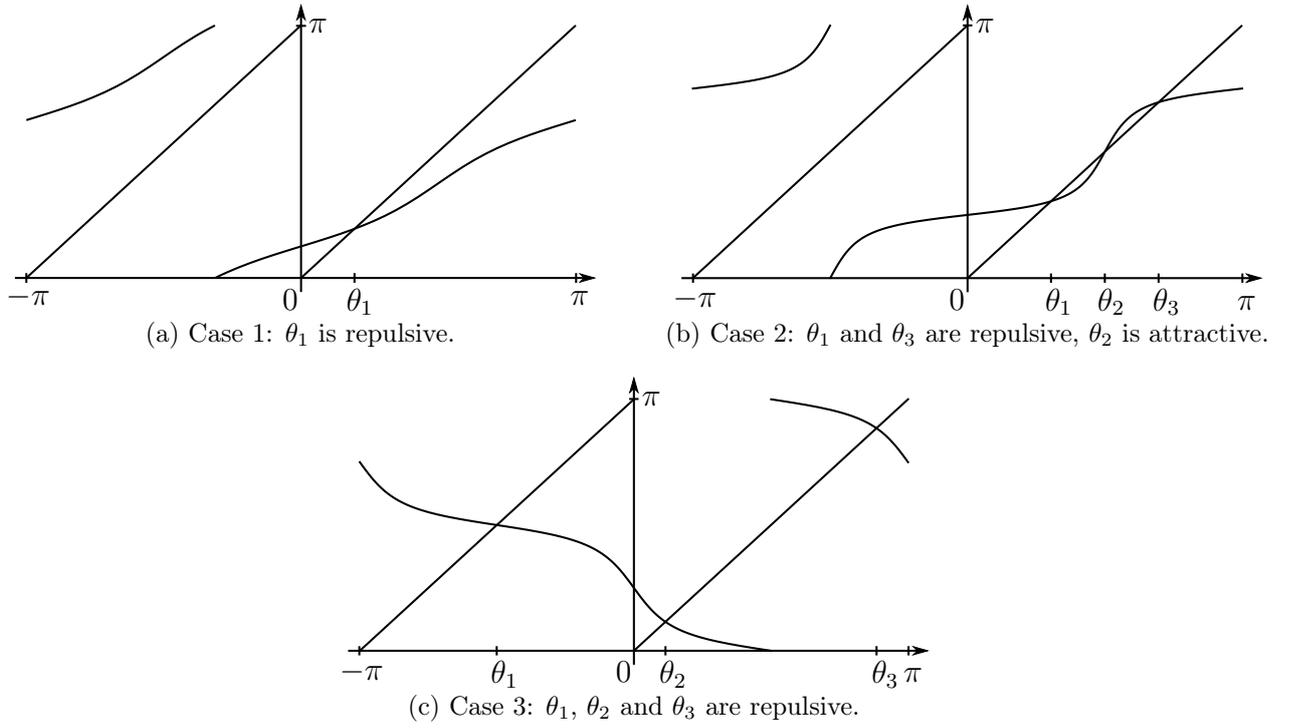

The behavior of $\phi_L$ in the three cases mentioned in the theorem is illustrated in Figure~\ref{Fig:eachCaseOfThmCas}. 
In order to show that this is the case we split the proof of Theorem \ref{T:thmcas} in several steps.
We start by proposing a suitable normal form for the vector field $X$. 

\begin{lem}\label{suit}
Let X be a linear vector field with a hyperbolic singularity. Then there exist $E,C>0$ and $\varphi\in [0,2\pi)$ such that, in some orthonormal basis, for every $\theta\in \R$, 
\begin{description}
\item{\bf C1} If the singularity is a focus, then 
$X
\left(
E \cos(\theta),  \sin(\theta)
\right)
=C \begin{pmatrix}
E \cos(\theta+\varphi)\\ \hphantom{E} \sin(\theta+\varphi)\end{pmatrix}$, with $E>1$.

\item{\bf C2} If the singularity is a node, then $X\left(
E \cos(\theta),  \sin(\theta)
\right)=C \begin{pmatrix}
E \cos(\theta-\varphi)\\ \hphantom{E}\sin(\theta+\varphi)\end{pmatrix}$, with $\varphi\in (-\pi/4,\pi/4) \pmod{\pi}$.

\item{\bf C3} If the singularity is a saddle, then $X \left(
E \cos(\theta),  \sin(\theta)
\right)=C \begin{pmatrix}
E \cos(\theta-\varphi)\\ \hphantom{E}\sin(\theta+\varphi)\end{pmatrix}$, with $\varphi\in (\pi/4,3\pi/4)\pmod{\pi}$.
\end{description}
\end{lem}

\begin{proof}
In each of the three cases, it is possible to find an orthonormal basis such that the matrix of the linear vector field is of the form 
\[
A=\begin{pmatrix}
a&b\\c&a
\end{pmatrix},
\]
with $b,c\neq 0$.
For $d=\sqrt{|bc|}$ and $E=\sqrt{|b/c|}$, we get for $\varepsilon=\pm 1$ and some $\varphi\in [0,2\pi)$
\[
A
=\sqrt{a^2+d^2}
\begin{pmatrix}
\hphantom{\frac{1}{E}} \cos\varphi &\varepsilon E \sin\varphi\\ \frac{1}{E} \sin\varphi & \hphantom{\varepsilon E}\cos \varphi
\end{pmatrix} .
\]

In case {\bf C1}, we know that the discriminant of $A$ is negative, so that $4\varepsilon \sin^2\varphi<0$, hence $\varepsilon=-1$. By further imposing $E>1$, if necessary by changing the orientation of the basis and by replacing $\varphi$ by $2\pi -\varphi$, we get the stated result.

In the two other cases, we know the discriminant of $A$ to be positive, so that $4\varepsilon \sin^2\varphi>0$, hence $\varepsilon=+1$.
Since $\det A=(a^2+d^2)(\cos^2\varphi-\sin^2\varphi)=(a^2+d^2)\cos 2\varphi$, we know that if the singularity is a node then $\cos 2 \varphi>0$, hence $\varphi\in (-\pi/4,\pi/4) \pmod{\pi}$, and if the singularity is a saddle then $\cos 2 \varphi<0$, hence $\varphi\in (\pi/4,3\pi/4) \pmod{\pi}$.
\end{proof}

In the following we assume an orthonormal basis of $\R^2$ as described in Lemma~\ref{suit} has been fixed.
\begin{lem}\label{lemtT}
Let $(X,Y)$ be a linear proto-line-field. 
Let $X$ be in one of the three normal forms of Lemma \ref{suit} 
and $\alpha=
\angle_{\mathrm{Eucl}}
\left[
	\begin{psmallmatrix}
	1 \\ 0
	\end{psmallmatrix},Y
\right]\in \R/2\pi\Z
$. 
Consider $F,G:\R/2\pi\Z\rightarrow \R/2\pi\Z$ smooth and such that 
$$
F(\theta)=
\angle_{\mathrm{Eucl}}
\left[
	\begin{pmatrix}
	1 \\ 0
	\end{pmatrix},	
	\begin{pmatrix}
	E \cos \theta 
	\\ 
	\hphantom{E}\sin \theta
	\end{pmatrix}
\right],\quad 
G(\theta)=
\angle_{\mathrm{Eucl}}
\left[
	\begin{pmatrix}
	1 \\ 0
	\end{pmatrix},	
	X(E \cos \theta, \sin \theta)
\right]
$$
for every $\theta\in \R/2\pi\Z$.

Then  $\theta_0\in \R/2\pi\Z$ is a fixed point of $\phi_L$ if and only if there exists $\theta_1\in \R/2\pi\Z$ such that  $\theta_0=F(\theta_1)$ and $2F(\theta_1)-G(\theta_1)=\alpha \pmod{2\pi}$.
Moreover $\theta_0$ is attractive if $(2F-G)'(\theta_1)<0$, and repulsive if $(2F-G)'(\theta_1)>0$.
\end{lem}

\begin{proof}
By its definition, $F$ is increasing and $F'>0$.
By definition of $F$, $G$ and $\alpha$, we have 
\begin{equation*}
\phi_L\circ F(\theta)=\frac{1}{2}G(\theta)+\frac{1}{2}\alpha\pmod{\pi}.
\end{equation*}
Fixed points of $\phi_L$ are then the images by $F$ of the solutions $\theta$ of
\begin{align}
\frac{1}{2}G(\theta)+\frac{1}{2}\alpha&=F(\theta) \pmod{\pi}\nonumber
\shortintertext{that is,}
2F(\theta)-G(\theta)&=\alpha \pmod{2\pi}, \label{eqEgal}
\end{align}
which proves the first part of the statement.
Since $F'>0$, the sign of $\phi_L'-1=(\phi_L-\mathrm{id})'$ is the sign of $(\phi_L\circ F-F)'=\frac{1}{2}(G-2F)'$, which proves the second part of the statement.
\end{proof}

The idea is now to study the variations of $2F-G$ to show that, depending on the values taken by $\alpha$ and the index of the singularity, we fall in one of the cases stated in Theorem \ref{T:thmcas}.

\begin{prop}\label{P:SigneDer}
Under the assumptions of Theorem \ref{T:thmcas}, if the singularity of $X$ has index 1, then there exist two constants $\Phi\in [-\pi,\pi)$ and $\kappa>0$ such that the sign of $(2F-G)'(\theta)$ is the sign of $\cos(2\theta+\Phi)+\kappa$. If the singularity of $X$ has index $-1$, then $(2F-G)'>0$ everywhere on $[-\pi,\pi)$.
\end{prop}
The proof of the proposition can be found in Appendix \ref{Appendix:proof_of_Proposition}. 

Before concluding the proof of Theorem \ref{T:thmcas}, let us emphasize the following two properties of $F$ and $G$.
\begin{description}
\item[P1] $\forall\theta\in (-\pi,\pi)$, $F(-\theta)=-F(\theta)$ and $G(-\theta)=-G(\theta)$;
\item[P2] $\forall\theta\in [-\pi,0)$, $F(\theta+\pi)-F(\theta)=\pi$ and $G(\theta+\pi)-G(\theta)=\pi$.
\end{description}

\begin{proof}[Proof of Theorem \ref{T:thmcas} when $X$ has a singularity of index 1.]

Let $\kappa>0$ and $\Phi\in [-\pi,\pi)$ be as in Proposition \ref{P:SigneDer}.

First consider the case where $\kappa\geq1$. The derivative $(2F-G)'$ is then always positive, except possibly at two  points in the case $\kappa=1$. Hence $2F-G$ is a bijection between $[-\pi,\pi)$ and its image. We claim that $2F-G\pmod{2\pi}$ is a bijection between $[-\pi,\pi)$ and $[0,2\pi)$. Equivalently we have to show that the image of $2F-G$ is an interval of length $2\pi$, which is immediate by application of $\mathbf{P2}$.

Thus we get the uniqueness in $[-\pi,\pi)$ of the solution $\theta$ of $2F(\theta)-G(\theta)=\alpha \pmod{2\pi}$. 
When $(2F-G)'(\theta)>0$ we deduce the repulsiveness from Lemma \ref{lemtT}. 
In the case $\kappa=1$ the set $\mathcal{E}_X$ is made of a single line through the origin, corresponding to the two values of  $\alpha$ for which $(2F-G)'(\theta)=0$ and  $2F(\theta)-G(\theta)=\alpha \pmod{2\pi}$. 

The case $\kappa \in (0,1)$ requires a further study of $2F-G$.
From Proposition \ref{P:SigneDer}, we know that $(2F-G)'(\theta)$ has the same sign as $\cos(2\theta+\Phi)+\kappa$. Let $\theta_0\in [0,\pi)$ be such that $\cos(2\theta_0+\Phi)+\kappa=0$ and $-2\sin(2\theta_0+\Phi)>0$, and let $\theta_1\in [\theta_0,\theta_0+\pi)$ be such that $\cos(2\theta_1+\Phi)+\kappa=0$ and $-2\sin(2\theta_1+\Phi)<0$. Then $\cos(2\theta+\Phi)+\kappa>0$ on $(\theta_0,\theta_1)$ and $\cos(2\theta+\Phi)+\kappa<0$ on $(\theta_1,\pi+\theta_0)$.
Since  $(2F-G)'$  is $\pi$-periodic, we let $x=\theta_1-\theta_0 \in (0,\pi)$ and up to replacing $\theta$ by $\theta-\theta_0$ (which corresponds to an orthonormal change of coordinates), we can assume that $(2F-G)'$ is positive on $(0,x)$ and negative on $(x,\pi)$. (See Figure~\ref{Fig:QualBehavior}.)
 
We are interested in characterizing the solutions of $2F-G=\alpha\pmod{2\pi}$ or, equivalently,
\[(2F-G)-(2F-G)(-\pi)=\alpha-(2F-G)(-\pi)\pmod{2\pi}.\]
Thus we can focus on the case $2F-G=\beta\pmod{2\pi}$, for some $\beta\in [-\pi,\pi)$ and $(2F-G)(-\pi)=0$.

Since $(2F-G)(0)=\pi$, and since $(2F-G)'$ is negative on $(x-\pi,0)$, there exists $y\in (0,x)$ such that $2F-G$ is increasing on $(-\pi,y-\pi)$ and $(2F-G)(y)=\pi$.
Moreover 
\[
\max_{\theta\in[0,\pi]} (2F-G)(\theta) \leq 2 \max_{\theta\in[-\pi,0]} F(\theta)-\min_{\theta\in[-\pi,0]}G(\theta).
\]
Since $F$ and $G$ are increasing, $\max_{\theta\in[-\pi,0]}F(\theta)=\pi$ and $\min_{\theta\in[-\pi,0]}G(\theta)=0$, so that \[\max_{\theta\in[0,\pi]} (2F-G)(\theta) \leq 2\pi.\]

Hence we know the behavior of the function $2F-G$, that is,
\begin{itemize}
\item $2F-G$ is increasing on $(-\pi, x-\pi)$, $(2F-G)(-\pi)=0$, $(2F-G)(x-\pi)\leq 2\pi$;
\item $2F-G$ is decreasing on $(x-\pi, 0)$, $(2F-G)(0)=\pi$;
\item $2F-G$ is increasing on $(0, x)$, and $(2F-G)(y)=2\pi$, then $(2F-G)(\theta)\leq 3\pi$ when $\theta \in [y,x]$;
\item $2F-G$ is decreasing on $[x, 2\pi]$, $(2F-G)(2\pi)=2\pi$.
\end{itemize}

\begin{figure}
\begin{center}
\begin{tikzpicture}[scale=1]

\tikzstyle{cadre}=[thin]
\tikzstyle{fleche}=[->,>=latex,thin]
\tikzstyle{nondefini}=[lightgray]

\def\Lrg{0.8}
\def\HtX{1}
\def\HtY{0.5}

\def\lignex{-0.5*\HtX}
\def\lignef{-1.5*\HtX}
\def\separateur{-0.5*\Lrg}

\def\gauche{-3*\Lrg}
\def\droite{12.5*\Lrg}

\def\haut{0.5*\HtX}
\def\bas{-2.5*\HtX-4*\HtY}

\node at (-1.7*\Lrg,0) {$\theta$};
\node at (0*\Lrg,0) {$-\pi$};
\node at (2*\Lrg,0) {$y-\pi$};
\node at (4*\Lrg,0) {$x-\pi$};
\node at (6*\Lrg,0) {$0$};
\node at (8*\Lrg,0) {$y$};
\node at (10*\Lrg,0) {$x$};
\node at (12*\Lrg,0) {$\pi$};

\node at (-1.7*\Lrg,-1*\HtX) {$(2F-G)'$};
\node at (0*\Lrg,-1*\HtX) {0};
\node at (2*\Lrg,-1*\HtX) {$+$};
\node at (4*\Lrg,-1*\HtX) {$0$};
\node at (5*\Lrg,-1*\HtX) {$-$};
\node at (6*\Lrg,-1*\HtX) {$0$};
\node at (8*\Lrg,-1*\HtX) {$+$};
\node at (10*\Lrg,-1*\HtX) {$0$};
\node at (11*\Lrg,-1*\HtX) {$-$};
\node at (12*\Lrg,-1*\HtX) {0};

\node  at (-1.7*\Lrg,{-2*\HtX+(-1)*\HtY}) {$2F-G$};
\node  at (-1.8*\Lrg,{-2*\HtX+(-2)*\HtY}) {$\pmod{2\pi}$};
\node (f1) at (0*\Lrg,{-2*\HtX+(-4)*\HtY}) {$0$};
\node (f2) at (2*\Lrg,{-2*\HtX+(-2)*\HtY}) {$\pi$};
\node (f3) at (4*\Lrg,{-2*\HtX+(-1)*\HtY}) {$<2\pi$};
\node (f4) at (6*\Lrg,{-2*\HtX+(-2)*\HtY}) {$\pi$};
\node(f5) at (8*\Lrg,{-2*\HtX+(0)*\HtY}) {$2\pi$};
\node(f6) at (8*\Lrg,{-2*\HtX+(-4)*\HtY}) {$0$};
\node (f7) at (10*\Lrg,{-2*\HtX+(-3)*\HtY}) {$<\pi$};
\node (f8) at (12*\Lrg,{-2*\HtX+(-4)*\HtY}) {$0$};

\draw[fleche] (f1) -- (f2);
\draw[fleche] (f2) -- (f3);
\draw[fleche] (f3) -- (f4);
\draw[fleche] (f4) -- (f5);
\draw[fleche] (f7) -- (f8);
\draw[fleche] (f6) -- (f7);

\draw[cadre] (\separateur,\haut) -- (\separateur,\bas);
\draw[cadre] (\gauche,\haut) rectangle  (\droite,\bas);
\draw[cadre] (\gauche,\lignex) -- (\droite,\lignex);
\draw[cadre] (\gauche,\lignef) -- (\droite,\lignef);
\end{tikzpicture}
\caption{Qualitative behavior of $2F-G$.\label{Fig:QualBehavior}}

\end{center}
\end{figure}

We can see that if $\alpha \not\in  [0,(2F-G)(x)] \cup [\pi,(2F-G)(-\pi+x)]+2k\pi$, $k\in \mathbb{Z}$, there is a unique repulsive solution.
If, instead, $\alpha \in  (0,(2F-G)(x)) \cup (\pi,(2F-G)(-\pi+x))+2k\pi$, $k\in \mathbb{Z}$, then there are three solutions, two repulsive and one attractive. Moreover, either the three solutions are all contained in $(y-\pi,y)$, with $(2F-G)'>0$ on the first and third solutions, or they all are in $[0,2\pi)\backslash [y-\pi,y]$ with $(2F-G)'>0$ on the first and second solutions (see Figure~\ref{Fig:QualBehavior}), which corresponds to the case 2 of Theorem \ref{T:thmcas}. Notice that the values of $\alpha$ which are not covered by this discussion are $0$, $(2F-G)(x)$, $\pi$, $(2F-G)(-\pi+x)$, which correspond to 
an exceptional set $\mathcal{E}_X$ made of two lines. 
\end{proof}

\begin{proof}[Proof of Theorem \ref{T:thmcas} when $X$ has a singularity of index $-1$.]
In this case we have that $G'<0$ and $F'>0$ on $[-\pi ,\pi)$. So $2F-G$, as a function from  $[-\pi ,\pi)$ to $\mathbb{R}$ is increasing and has total variation $6\pi$.  Hence there exists $y\in(-\pi,0)$ such that $(2F-G)-(2F-G)(-\pi)$ is a bijection from $(-\pi,y)$ onto $(0,2\pi)$; it exists $x\in(y,\pi)$ such that $(2F-G)-(2F-G)(-\pi)$ is a bijection from $(y,x)$ onto $(2\pi,4\pi)$; and again $(2F-G)-(2F-G)(-\pi)$ is a bijection from $(x,\pi)$ onto $(4\pi,6\pi)$.
Hence we have found that there are three solutions to (\ref{eqEgal}).

Since $(2F-G)(\theta+\pi)=(2F-G)(\theta)+\pi\pmod{2\pi}$, we know that each half-line where the line field is orthogonal to the line of the position is in the opposite direction to one of the solutions of equation (\ref{eqEgal}). By monotonicity of $2F-G$, we then conclude that the three fixed points cannot be in the same half-plane.
\end{proof}

Theorem \ref{T:thmcas} motivates the following definition.
\begin{defn}
A linear hyperbolic proto-line-field $L$ is said to be \emph{\superhyper} if for every $\theta_0\in \R$ such that $\phi_L(\theta_0)=\theta_0\pmod{\pi}$, one has $\frac{d\phi_L}{d\theta}(\theta_0)\neq 1$.
\end{defn}

\section{Linearization, blow-up and proof of Theorem~\ref{T:MainTheorem}}
\label{Section:Linearization}

The goal of this section is to prove the topological equivalence of a \superhyper\ proto-line-field at a hyperbolic singularity and its linearization. 
As a direct consequence, we get a proof of \MT.

\subsection{Blow-up}
\label{SubSection:Blowup}

Proposition \ref{P:blowup} below is the main technical step in the construction  of the topological equivalence.
It provides a blow up of a hyperbolic singularity of a proto-line-field.
The blow up sends the singularity into a line and allows to describe locally the line field by means of a vector field on a strip containing such line.

In what follows set $\mathrm{Pol}$ to be the Riemannian metric on $\R^+\times \R$ defined  by $\mathrm{Pol}(r,\theta)=dr^2+r^2d\theta^2$ and
recall that a Riemannian metric can always be diagonalized at a point by a suitable choice of coordinates.

\begin{prop}\label{P:blowup}
Let $L=(X,Y)$ be a proto-line-field on $(M,g)$ with a hyperbolic singularity at $p\in M$. Fix a system of coordinates $(x,y)$ such that $p=(0,0)$, $g(0,0)=\mathrm{id}$ and assume that $(x,y)$ defines a diffeomorphism between a neighborhood of $p$ and the ball of center the origin and radius $\delta$, for some $\delta>0$ such that $p$ is the only singularity of $L$ on the ball.
Assume that $Y(0,0)\not =0$ and consider the linear proto-line-field $\bar{L}=(DX(0,0),Y(0,0))$.

For every $(r,\theta)\in (0,\delta)\times \R$ 
let $\phi_L(r,\theta)$ and  $\phi_{\bar{L}}(\theta)$ in $\mathbb{R}\slash \pi \mathbb{Z}$ be defined by
$$
\phi_L(r,\theta)=
\angle_{g}
\left[
	\begin{pmatrix}
	1 \\ 0
	\end{pmatrix},B(L(r\cos\theta,r\sin\theta))
\right]
\text{ and } 
\phi_{\bar{L}}(\theta)=
\angle_{\mathrm{Eucl}}
\left[
	\begin{pmatrix}
	1 \\ 0
	\end{pmatrix},B(\bar{L}(r\cos\theta,r\sin\theta))
\right].
$$
Then there exists a $C^1$ function $\widetilde{\phi}_L:(-\delta,\delta)\times \R \rightarrow \R$ such that for every $(r,\theta)\in (0,\delta)\times \R$
\begin{align*}
&\phi_L(r,\theta)=\widetilde{\phi}_L(r,\theta)\pmod{\pi},
\\
&\widetilde{\phi}_L(-r,\theta)=\widetilde{\phi}_L(r,\theta)
\\
&\widetilde{\phi}_L(r,\theta+4\pi)=\widetilde{\phi}_L(r,\theta)\pmod{2\pi},
\end{align*}
 and such that the vector field $P$ on $(-\delta,\delta)\times (\R/4\pi \Z)$ given by 
\begin{equation}\label{E:defP}
P(r,\theta)=
\begin{pmatrix}
	r\cos(\widetilde{\phi}_L(r,\theta)-\theta)
	\\
	\sin(\widetilde{\phi}_L(r,\theta)-\theta)
\end{pmatrix}
\end{equation}
is $C^1$ and satisfies, for all $(r,\theta)\in (0,\delta)\times (\R/4\pi \Z)$,
\begin{equation}\label{E:defAngleP}
\angle_{\mathrm{Pol}}
\left[
	\begin{pmatrix}
	1\\0
	\end{pmatrix},P(r,\theta)
\right]
=
\angle_{g}
\left[
	\begin{pmatrix}
	\cos \theta \\ \sin \theta
	\end{pmatrix},B(L(r\cos\theta,r\sin\theta))
\right]
\pmod{\pi}.
\end{equation}
The singularities of $P$ in $(-\delta,\delta)\times (\R/4\pi \Z)$ are the points $(0,\theta_0)$ such that $\phi_{\bar{L}}(\theta_0)=\theta_0\pmod{\pi}$. Moreover,  if $\theta_0$ is a repulsive (respectively, attractive) fixed point of $\phi_{\bar{L}}$ then the singularity $(0,\theta_0)$ of $P$ is a saddle (respectively, a node). 
\end{prop}

\begin{proof}
Since $\delta>0$ has been chosen small enough so that the only singularity of $L$ in $\{x^2+y^2\leq \delta^2\}$ is in $(0,0)$, then $\phi_L :(0,\delta)\times  \R\rightarrow \R/\pi\Z$ can be lifted as a smooth function $\widetilde{\phi}_L:(0,\delta)\times \R\rightarrow \R$.

Lemma \ref{L:blowup}, found in appendix \ref{Appendix:Extension_of_the direction}, states that the limits
$$
\phi_L(r,\theta) \underset{r\rightarrow 0}{\longrightarrow} \phi_{\bar{L}}(\theta),
\quad
\frac{\partial\phi_L}{\partial\theta}(r,\theta) \underset{r\rightarrow 0}{\longrightarrow} \frac{d\phi_{\bar{L}}}{d\theta}(\theta),\text{ and }\quad
\frac{\partial\phi_L}{\partial r}(r,\theta) \underset{r\rightarrow 0}{\longrightarrow} 0,
$$
hold true, thus proving that $\widetilde{\phi}_L$ admits a $C^1$ extension on $[0,\delta)\times \R$.
We can then symmetrically extend $\widetilde{\phi}_L$ onto  $(-\delta,\delta)\times \R$ by setting $\widetilde{\phi}_L(r,\theta)=\widetilde{\phi}_L(-r,\theta)$ for any $r\in (-\delta,0)$. By construction, $\widetilde{\phi}_L$ is $C^1$.

From this construction we deduce that $P$, defined as in \eqref{E:defP}, is $C^1$ on $(-\delta,\delta)\times (\R/4\pi \Z)$ and that it vanishes exclusively at the points $(0,\theta_0)$, where $\phi_{\bar{L}}(\theta_0)=\theta_0\pmod{\pi}$. Its definition further implies relation \eqref{E:defAngleP}.

The differential of $P$ at a point $(r,\theta)\in (-\delta,\delta)\times (\R/4\pi \Z)$ is given by 
\[
DP(r,\theta)=
\begin{pmatrix}
\cos\left(\widetilde{\phi}_L-\theta\right)+r \frac{\partial\widetilde{\phi}_L}{\partial r}  \sin\left(\widetilde{\phi}_L-\theta\right) 
&
r\left(  \frac{\partial\widetilde{\phi}_L}{\partial\theta} -1\right)\sin\left(\widetilde{\phi}_L-\theta\right)
\\
\frac{\partial\widetilde{\phi}_L}{\partial r}  \cos\left(\widetilde{\phi}_L-\theta\right) 
&
\left(  \frac{\partial\widetilde{\phi}_L}{\partial\theta} -1\right) \cos\left(\widetilde{\phi}_L-\theta\right)
\end{pmatrix}.
\]
Thus, at a singularity $(0,\theta_0)$, where $\phi_{\bar{L}}(\theta_0)=\theta_0\pmod{\pi}$, the differential of $P$  is given by
\[
DP(0,\theta_0)=
\cos\left(\widetilde{\phi}_L(0,\theta_0)-\theta_0\right)
\begin{pmatrix}
1 & 0
\\
0
&
 \frac{d\widetilde{\phi}_{\bar{L}}}{d\theta} -1
\end{pmatrix}.
\]
Since  $\phi_{\bar{L}}(\theta_0)=\theta_0\pmod{\pi}$, moreover, we have $\cos\left(\widetilde{\phi}_L(0,\theta_0)-\theta_0\right)=\pm 1$. 
Therefore,  if $\theta_0$ is attractive then $\frac{\partial\phi_{\bar{L}}}{\partial\theta} >1$ and the singularity is a node, while  if $\theta_0$ is repulsive then $\frac{\partial\phi_{\bar{L}}}{\partial\theta} >1$ and the singularity is a saddle.
\end{proof}

The proto-line-field $\bar{L}$ in Proposition \ref{P:blowup} plays the role of the linearization of $L$. This motivates the following definition.
\begin{defn}\label{d:linearization}
Let $L=(X,Y)$ be a proto-line-field on $(M,g)$ with a singularity at $p\in M$ such that $Y(p)\neq0$. Fix a system of coordinates $(x,y)$ such that $p=(0,0)$, $g(0,0)=\mathrm{id}$. Then we call \emph{linearization of $L$ at $p$} the 
 linear proto-line-field $\bar{L}=(DX(0,0),Y(0,0))$.
\end{defn}

Notice that the condition $g(0,0)=\mathrm{id}$ defines uniquely $\bar{L}$ up to an orthogonal transformation of $\R^2$. In particular, the fact that $\bar{L}$ satisfies one of the properties \ref{PT:1}, \ref{PT:2} or \ref{PT:3} of Theorem \ref{T:thmcas} depends only on $L$.

\subsection{Proof of \MT}
\label{SubSection:Proof_of_MT}

The proof of \MT\  is based on the following proposition.
\begin{prop}\label{P:TopEqMain}
Let $L$ and $L'$ be two proto-line-fields on $(M,g)$ and $(M',g')$ respectively. Let $p$ and $p'$ be two hyperbolic singularities of $L$ and $L'$ respectively. 
Let $\bar{L}$ and $\bar{L}'$ be the corresponding linearizations and assume that one of the properties \ref{PT:1}, \ref{PT:2} or \ref{PT:3} of Theorem \ref{T:thmcas} is satisfied both by $\bar{L}$ and $\bar{L}'$. Then $L$ and $L'$ are topologically equivalent at $p$ and $p'$.
\end{prop}
The proof of Proposition \ref{P:TopEqMain} is given in the next section. A first consequence of the proposition is the following corollary.
\begin{cor}
Let $L$ be a proto-line-field on $(M,g)$. Let $p$  be a hyperbolic singularity of $L$.
Let $\bar{L}$ be the corresponding linearization. If $\bar{L}$ is \superhyper\ then $L$ and $\bar{L}$ are topologically equivalent at $p$ and $0$.
\end{cor}

\MT\ can now be obtained by combining Theorem \ref{T:thmcas}, Proposition \ref{P:TopEqMain} and Thom's transversality theorem.
\begin{proof}[Proof of \MT]
As a consequence of Thom's transversality theorem,
generically  with respect to $(X,Y)$ in $\mathcal{W}^1(M)$,  every singularity of the proto-line-field $L=(X,Y)$ is hyperbolic and its linearization 
is \superhyper.
Notice that the Lemon (respectively, Monstar, Star) proto-line-field satisfies property \ref{PT:1} (respectively, 
\ref{PT:2}, 
\ref{PT:3})
of Theorem~\ref{T:thmcas}. It then follows from Proposition \ref{P:TopEqMain} that all singularities of a generic proto-line-field are Darbouxian.
\end{proof}

\subsection{Construction of the topological equivalence}
\label{SubSection:Construction_of_the_topological_equivalence}

In order to prove Proposition~\ref{P:TopEqMain}, let us first focus on the following
two lemmas, which yield conditions for the existence of homeomorphisms preserving integral manifolds of proto-line-fields around singularities.

\begin{lem}\label{L:TopEqLift}
Let $L$ and $L'$ be two proto-line-fields on $(M,g)$ and $(M',g')$ respectively. Let $p$ and $p'$ be two hyperbolic singularities of $L$ and $L'$ respectively. 
Let $\bar{L}$ and $\bar{L}'$ be the corresponding linearizations and assume that one of the properties \ref{PT:1}, \ref{PT:2} or \ref{PT:3} of Theorem~\ref{T:thmcas} is satisfied both by $\bar{L}$ and $\bar{L}'$. 
Let $\delta>0$ (respectively, $\delta'>0$) and $P$ (respectively, $P'$) be the vector field on $(-\delta,\delta)\times (\R/4\pi \Z)$ (respectively $(-\delta',\delta')\times (\R/4\pi \Z)$) introduced in Proposition~\ref{P:blowup}. 

Then there exist 
two neighborhoods $V$ and $V'$ of $\{0\}\times (\R/4\pi \Z)$, invariant under the  translation $T:(r,\theta)\mapsto (r,\theta+2\pi)$,
and an homeomorphism $h:V \rightarrow V'$ such that $h$ maps the integral lines of $P$ onto the integral lines of $P'$ and $h\circ T=T\circ h$.
\end{lem}

\begin{proof}
Choose $\eta>0$ small enough so that $P$ has no cycle nor integral curve with both ends at a saddle  in $(-\eta,0)\times (\R/4\pi\Z)$.
We are interested in studying the skeleton of $P$, {\it i.e.}, the union $S$ of the set of zeros of $P$ and integral curves in $(-\eta,\eta)\times (\R/4\pi\Z)$ 
that reach a saddle singularity of $P$ at one of its ends (or at both of them). 
The set $((-\eta,\eta)\times (\R/4\pi\Z))\backslash S$ has exactly twice as many connected components as $P$ has saddles in $(-\eta,\eta)\times (\R/4\pi\Z)$, and four times as many as $\phi_{\bar{L}}$ has repulsive fixed points in $\R/2\pi\Z$. (See Figure~\ref{Fig:AnnulusBlowUp}.)

Let $\{C_i\mid i\in I\}$ be the set of connected components of $((-\eta,\eta)\times (\R/4\pi\Z))\backslash S$. 

The border $\partial C$ of a cell $C$ is the union of a segment  of the type
$\{0\}\times [\theta_1,\theta_2]$, of an arc of  $\{\pm\eta \}\times (\R/4\pi\Z)$, and of
two integral curves $\gamma_1$, $\gamma_2$  of $P$ that join $(0,\theta_1)$ and $(0,\theta_2)$.

If there is no attractive fixed point of $\phi_{\bar{L}}$ in the interval $(\theta_1,\theta_2)$ then we can find an integral line $\gamma$ of $P$ that is arbitrarily close to  $\gamma_1\cup(\{0\}\times [\theta_1,\theta_2])\cup\gamma_2$ (see Figure~\ref{Fig:CC1}). Then we can assume that the vector field $P$ is transverse to $\{\pm\eta \}\times (\R/4\pi\Z)$ between 
$\gamma$ and $\gamma_1\cup(\{0\}\times [\theta_1,\theta_2])\cup\gamma_2$, and that it is topologically equivalent  on this subset to the parallel vector field 
$$
\begin{pmatrix}
1\\0
\end{pmatrix}
\text{ on }[0,1]\times [0,1].
$$

If there is an attractive fixed point of $\phi_{\bar{L}}$ in the interval $(\theta_1,\theta_2)$ then we can find $0<\eta_1<\eta$ so that $P$ is transverse to $\{\pm\eta_1 \}\times  (\R/4\pi\Z)$ (see Figure~\ref{Fig:CC2}) and $P$ is topologically equivalent on the intersection of the cell with $(-\eta_1,\eta_1)\times (\R/4\pi\Z)$ to 
$$
\begin{pmatrix}
x\\y
\end{pmatrix}
\text{ on }[0,1]\times [0,1].
$$

Since $L$ and $L'$ satisfy the same property   \ref{PT:1}, \ref{PT:2} or \ref{PT:3} of Theorem \ref{T:thmcas}, their skeletons are homeomorphic. Since the directions of $P$ and $P'$ are $2\pi$-periodic, the construction above  leads to a topological equivalence between $P$ and $P'$ as in the statement of the lemma.
\end{proof}

\begin{figure}[!htb]
\begin{center}
\begin{minipage}{0.45\linewidth}
\subfloat[		Case \ref{PT:1}\label{Fig:Blowup1} ]{\includegraphics[width=\linewidth]{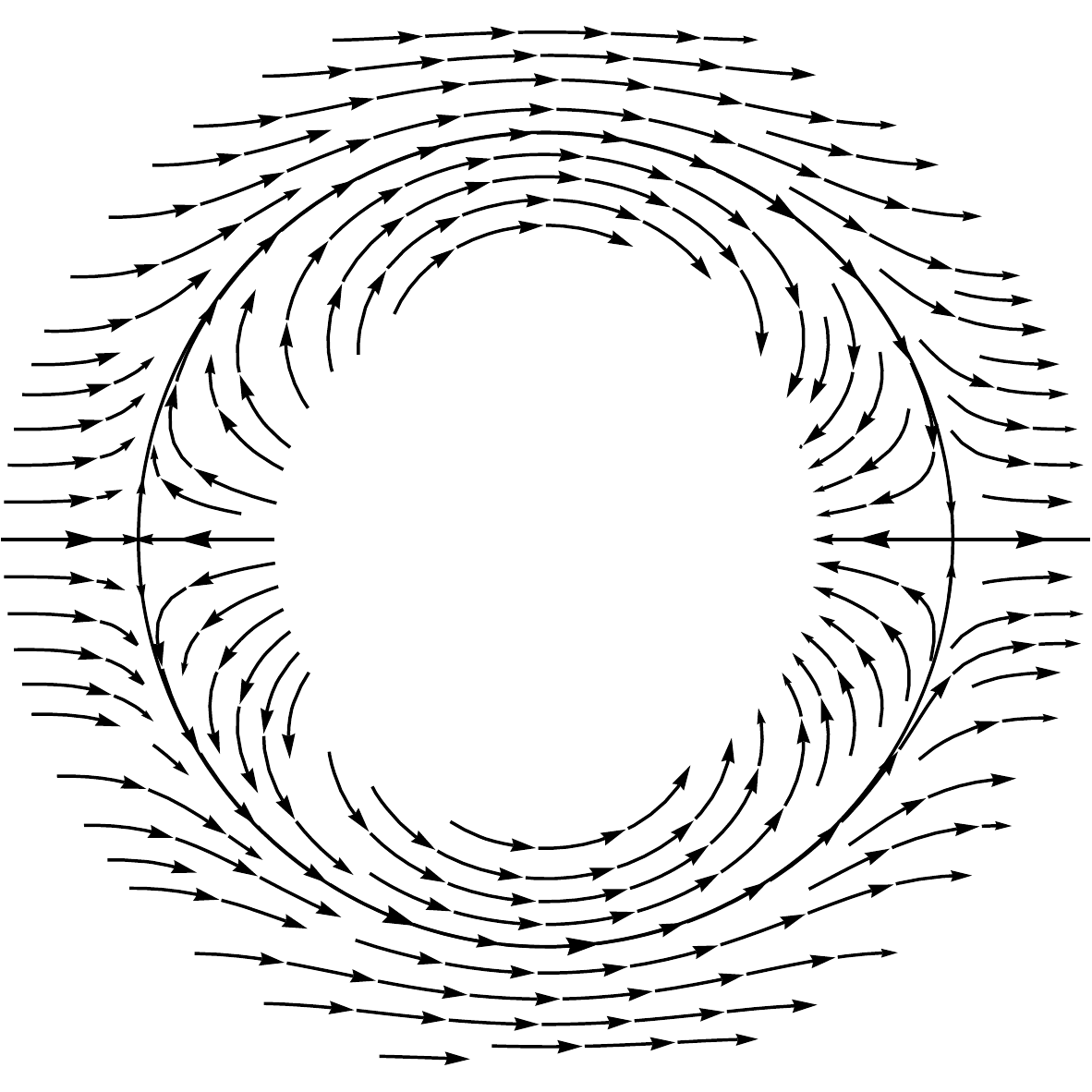}}
\end{minipage}
\hspace{0.05\linewidth}
	\begin{minipage}{0.45\linewidth}
\subfloat[Case \ref{PT:2}\label{Fig:Blowup2}  ]{\includegraphics[width=\linewidth]{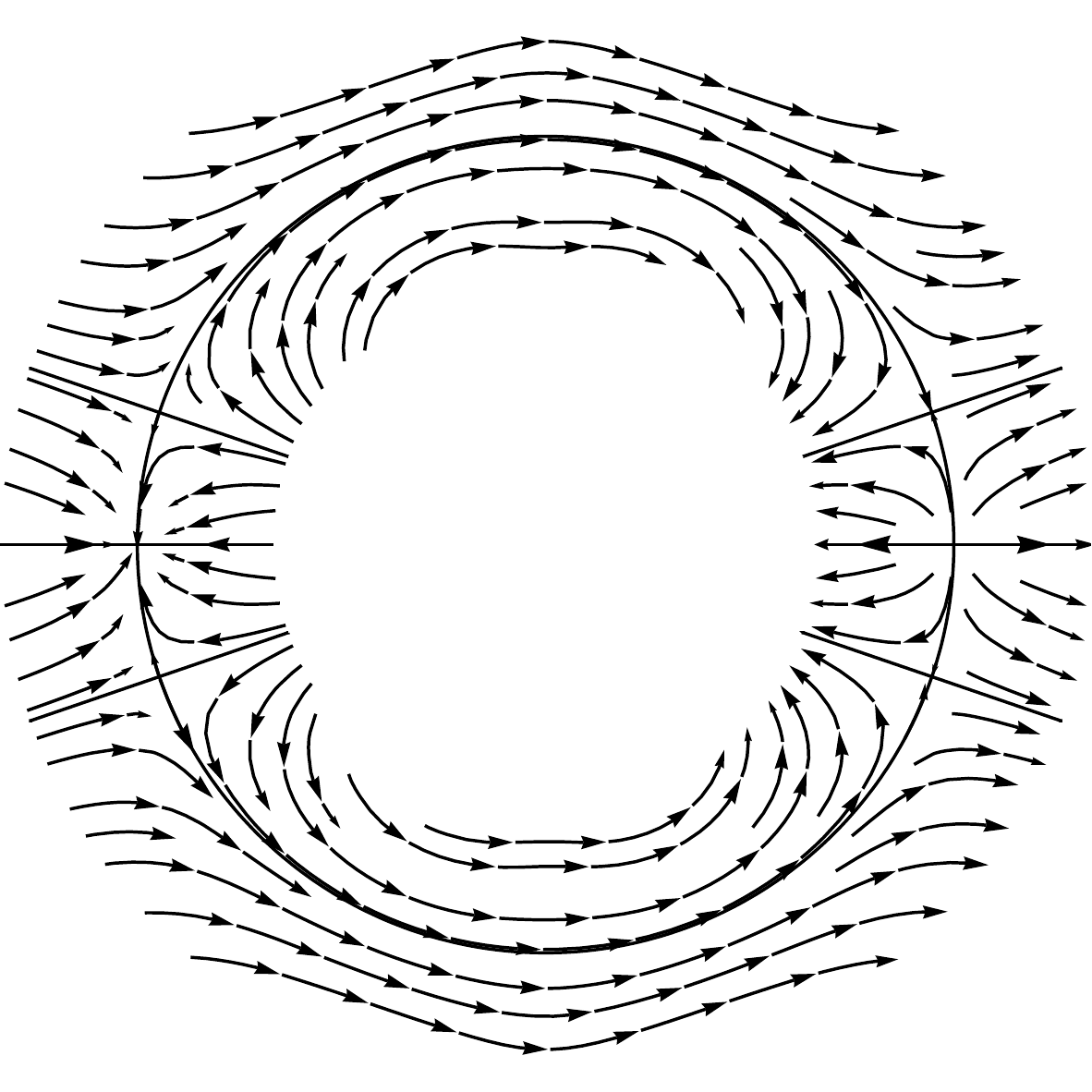}}
\end{minipage}	

\vspace{0.5\baselineskip}

\begin{minipage}{0.45\linewidth}
\subfloat[Case \ref{PT:3}\label{Fig:Blowup3} ]{\includegraphics[width=\linewidth]{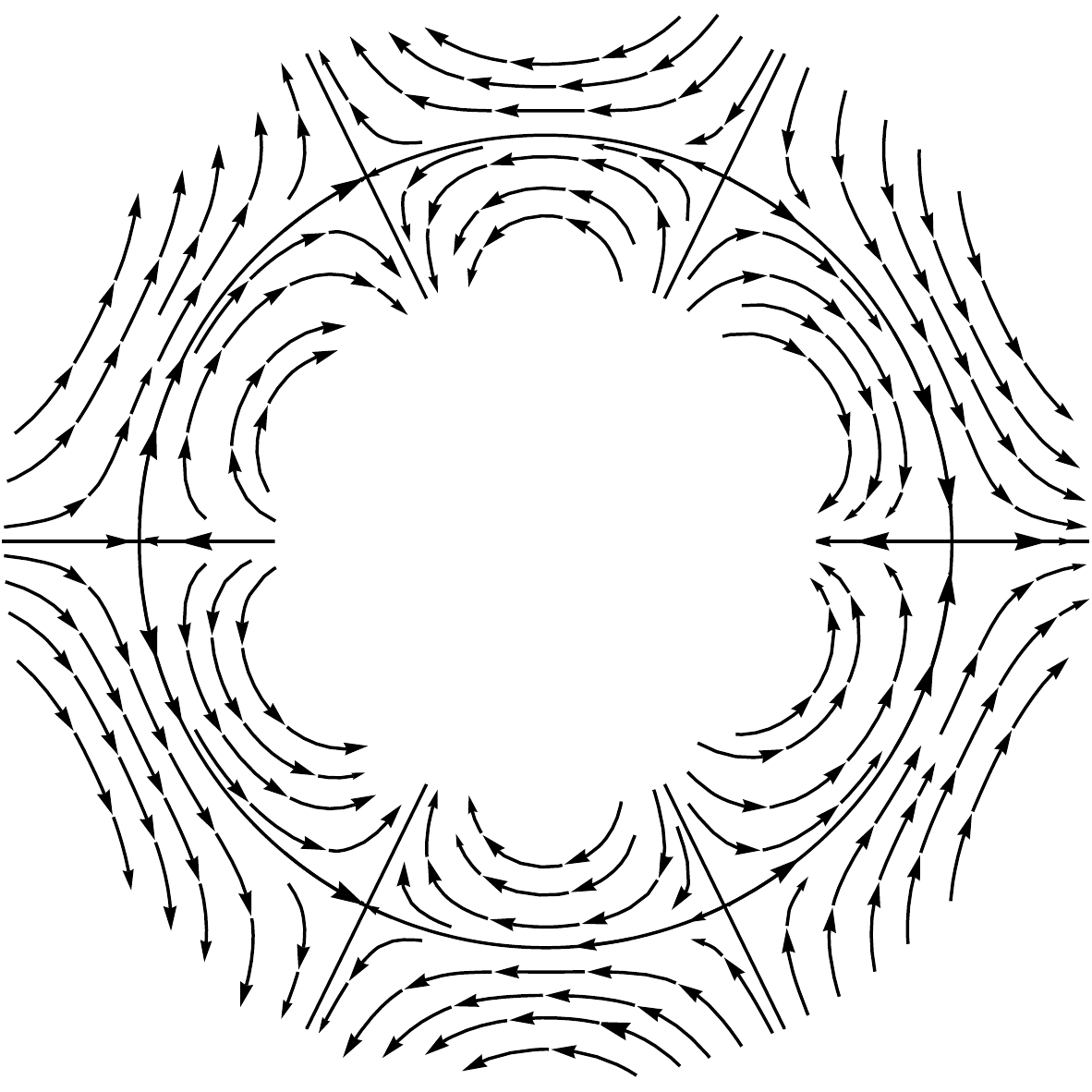}} 
\end{minipage}

\caption[Examples of $P$ in each case of Theorem \ref{T:thmcas}, where $(-\delta_0,\delta_0)\times (\R/4\pi\Z)$ is represented as an annulus. The figures \ref{Fig:Blowup1}, \ref{Fig:Blowup2} and  \ref{Fig:Blowup3} where respectively computed with and , , .]{Examples of $P$ in each case of Theorem \ref{T:thmcas}, where $(-\delta_0,\delta_0)\times (\R/4\pi\Z)$ is represented as an annulus. The figures \ref{Fig:Blowup1}, \ref{Fig:Blowup2} and  \ref{Fig:Blowup3} where respectively computed with $Y=\begin{pmatrix}
1\\0
\end{pmatrix}$
and
$X=\begin{pmatrix}
x\\10y/11
\end{pmatrix}$, 
$X=\begin{pmatrix}
x\\5y
\end{pmatrix}$, 
$X=\begin{pmatrix}
x\\-3y
\end{pmatrix}$, respectively.
\label{Fig:AnnulusBlowUp}
}
\end{center}
\end{figure}
\begin{figure}[!htb]
\begin{center}
         \subfloat[First case: there is no attractive fixed point of $\phi_{\bar{L}}$ between $\theta_1$ and $\theta_2$. The topological equivalence is defined on a subset of $C_i$ bounded by the curve $\gamma$.\label{Fig:CC1}]{
\begingroup%
  \makeatletter%
  \providecommand\color[2][]{%
    \errmessage{(Inkscape) Color is used for the text in Inkscape, but the package 'color.sty' is not loaded}%
    \renewcommand\color[2][]{}%
  }%
  \providecommand\transparent[1]{%
    \errmessage{(Inkscape) Transparency is used (non-zero) for the text in Inkscape, but the package 'transparent.sty' is not loaded}%
    \renewcommand\transparent[1]{}%
  }%
  \providecommand\rotatebox[2]{#2}%
  \ifx\svgwidth\undefined%
    \setlength{\unitlength}{0.3\textwidth}%
    \ifx\svgscale\undefined%
      \relax%
    \else%
      \setlength{\unitlength}{\unitlength * \real{\svgscale}}%
    \fi%
  \else%
    \setlength{\unitlength}{\svgwidth}%
  \fi%
  \global\let\svgwidth\undefined%
  \global\let\svgscale\undefined%
  \makeatother%
  \begin{picture}(1,2)%
    \put(0,0){\includegraphics[width=\unitlength,page=1]{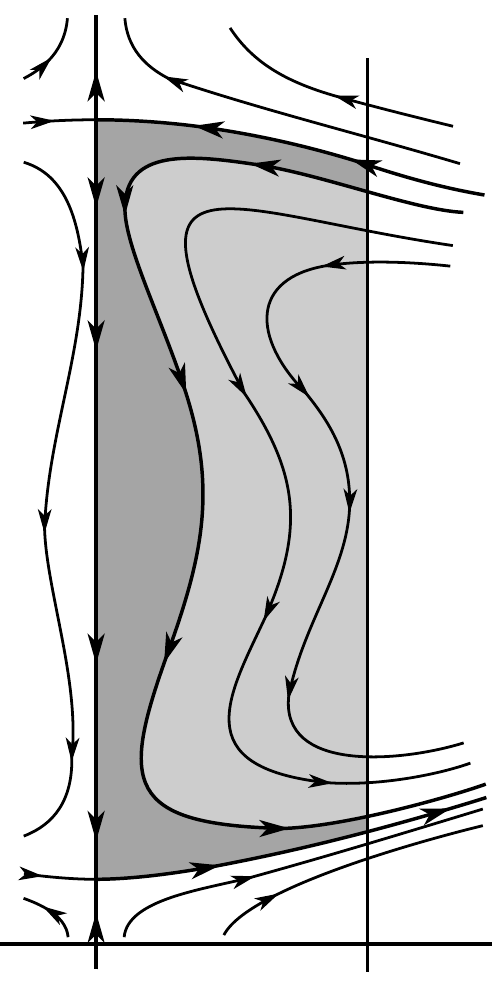}}%
    \put(0.3,0.235){\color[rgb]{0,0,0}\makebox(0,0)[lb]{$\gamma_1$}}%
    \put(0.45,0.85){\color[rgb]{0,0,0}\makebox(0,0)[lb]{$C_i$}}%
    \put(0.35,1){\color[rgb]{0,0,0}\makebox(0,0)[lb]{$\gamma$}}%
    \put(0.21,0.01){\color[rgb]{0,0,0}\makebox(0,0)[lb]{0}}%
    \put(-0.025,0.24){\color[rgb]{0,0,0}\makebox(0,0)[lb]{\colorbox{white}{\makebox[0.75cm][c]{$(0,\theta_1)$}}}}
    \put(-0.025,1.61){\color[rgb]{0,0,0}\makebox(0,0)[lb]{\colorbox{white}{\makebox[0.75cm][c]{$(0,\theta_2)$}}}}
    \put(0.76,0.01){\color[rgb]{0,0,0}\makebox(0,0)[lb]{$\eta$}}%
    \put(0.25,1.76611713){\color[rgb]{0,0,0}\makebox(0,0)[lb]{$\gamma_2$}}%
  \end{picture}%
\endgroup%
         }      
        \hspace{0.1\textwidth}
        \subfloat[Second case: there is an attractive fixed point $\theta_0$ of $\phi_{\bar{L}}$ between $\theta_1$ and $\theta_2$. For $\eta$ small enough the topological equivalence is defined on the entire cell $C_i$.\label{Fig:CC2}]{
\begingroup%
  \makeatletter%
  \providecommand\color[2][]{%
    \errmessage{(Inkscape) Color is used for the text in Inkscape, but the package 'color.sty' is not loaded}%
    \renewcommand\color[2][]{}%
  }%
  \providecommand\transparent[1]{%
    \errmessage{(Inkscape) Transparency is used (non-zero) for the text in Inkscape, but the package 'transparent.sty' is not loaded}%
    \renewcommand\transparent[1]{}%
  }%
  \providecommand\rotatebox[2]{#2}%
  \ifx\svgwidth\undefined%
    \setlength{\unitlength}{0.3\textwidth}%
    \ifx\svgscale\undefined%
      \relax%
    \else%
      \setlength{\unitlength}{\unitlength * \real{\svgscale}}%
    \fi%
  \else%
    \setlength{\unitlength}{\svgwidth}%
  \fi%
  \global\let\svgwidth\undefined%
  \global\let\svgscale\undefined%
  \makeatother%
  \begin{picture}(1,2)%
    \put(0,0){\includegraphics[width=\unitlength,page=1]{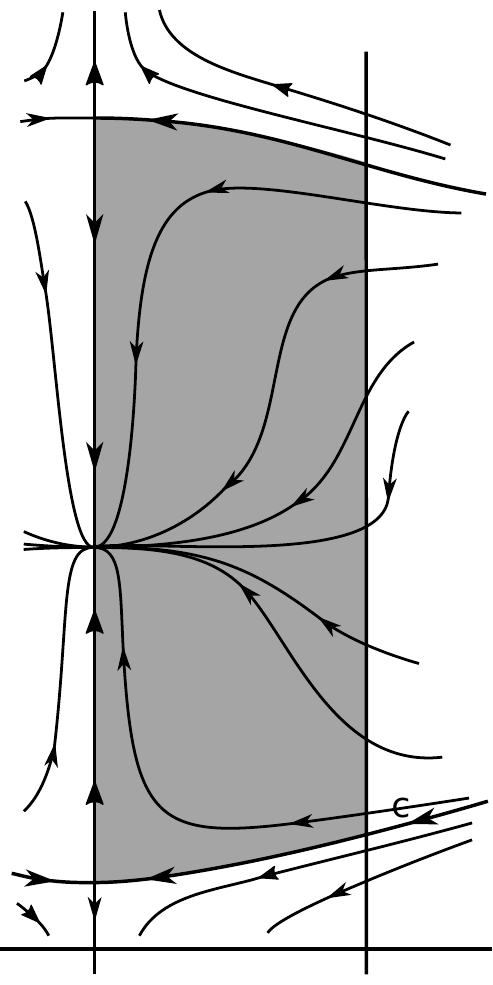}}%
    \put(0.76,0.0){\color[rgb]{0,0,0}\makebox(0,0)[lb]{$\eta$}}%
    \put(0.21951737,0.0){\color[rgb]{0,0,0}\makebox(0,0)[lb]{$0$}}%
    \put(0.25,0.25){\color[rgb]{0,0,0}\makebox(0,0)[lb]{$\gamma_1$}}%
    \put(0.4,0.5){\color[rgb]{0,0,0}\makebox(0,0)[lb]{$C_i$}}%
    \put(0.25,1.765){\color[rgb]{0,0,0}\makebox(0,0)[lb]{$\gamma_2$}}%
    \put(-0.03,1.61){\color[rgb]{0,0,0}\makebox(0,0)[lb]{\colorbox{white}{\makebox[0.75cm][c]{$(0,\theta_2)$}}}}%
    \put(-0.03,0.237){\color[rgb]{0,0,0}\makebox(0,0)[lb]{\colorbox{white}{\makebox[0.75cm][c]{$(0,\theta_1)$}}}}%
    \put(-0.07,0.75){\color[rgb]{0,0,0}\makebox(0,0)[lb]{\colorbox{white}{\makebox[0.75cm][c]{$(0,\theta_0)$}}}}%
  \end{picture}%
\endgroup%
} 
\caption{Representation of the two types of connected components of $((-\eta,\eta)\times (\R/4\pi\Z))\backslash S$.}
\end{center}
\end{figure}

\begin{lem}\label{L:TopEqPolar}
Let $L$ be a proto-line-field on $(M,g)$. Let $p$ be a hyperbolic singularity of $L$. 
Assume that the linearization $\bar{L}$  of $L$ at $p$ is \superhyper. Let $delta>0$, the system of coordinates $(x,y)$, and the vector field $P$  on $(-\delta,\delta)\times (\R/4\pi \Z)$ be defined as in Proposition \ref{P:blowup}. 
Then the application
$$
\begin{array}{rccl}
\psi:&(0,\delta) \times(\R/4\pi \Z)&\longrightarrow &\{0<x^2+y^2<\delta^2\}
\\
&(r,\theta)&\longmapsto&(x,y)=(r\cos\theta,r\sin\theta)
\end{array}
$$ 
is a local diffeomorphism that maps the integral lines of $P$ onto the integral manifolds of $L$.
\end{lem}

\begin{proof}
The map $\psi$ is a local diffeomorphism since the differential of $\psi$ 
is given by
$$
D\psi (r,\theta)=
\begin{pmatrix}
\cos\theta & -r\sin\theta
\\
\sin\theta & r\cos \theta
\end{pmatrix}.
$$
Hence, 
\begin{align*}
\psi_*(P(r,\theta))=\begin{pmatrix}
\cos\theta & -r\sin\theta
\\
\sin\theta & r\cos \theta
\end{pmatrix}
\begin{pmatrix}
	r\cos(\widetilde{\phi}_L(r,\theta)-\theta)
	\\
	\sin(\widetilde{\phi}_L(r,\theta)-\theta)
\end{pmatrix}
&=
\begin{pmatrix}
	r\cos(\widetilde{\phi}_L(r,\theta))
	\\
	r\sin(\widetilde{\phi}_L(r,\theta))
\end{pmatrix}.
\end{align*}
Thus, the locally defined vector field $\psi_*P$ is parallel to $L$, 
concluding the proof of the lemma.
\end{proof}

We are now ready to prove Proposition \ref{P:TopEqMain}.

\begin{proof}[Proof of Proposition \ref{P:TopEqMain}]
Let $\delta>0$ (respectively, $\delta'>0$) and $P$ (respectively, $P'$) be the vector field on $(-\delta,\delta)\times (\R/4\pi \Z)$ (respectively $(-\delta',\delta')\times (\R/4\pi \Z)$) introduced in Proposition~\ref{P:blowup}. 
Consider $V$, $V'$ and the homeomorphism $h:V\rightarrow V'$
introduced in Lemma~\ref{L:TopEqLift}. Finally, following Lemma~\ref{L:TopEqPolar}, let $\psi$ and $\psi'$ be the local diffeomorphisms 
$$
\begin{array}{rccl}
\psi:&(0,\delta)\times \R/4\pi\Z&\longrightarrow &D_p(\delta)=\{0<x^2+y^2<\delta^2\}
\\
&(r,\theta)&\longmapsto&(x,y)=(r\cos\theta,r\sin\theta)
\end{array}
$$ 
and 
$$
\begin{array}{rccl}
\psi':&(0,\delta')\times \R/4\pi\Z &\longrightarrow &D'_{p'}(\delta')=\{0<x'^2+y'^2<\delta'^2\}
\\
&(r,\theta)&\longmapsto&(x',y')=(r\cos\theta,r\sin\theta).
\end{array}
$$ 

The translation $T:(r,\theta)\mapsto (r,\theta+2\pi)$ induces  a natural fiber bundle structure $\pi:V\rightarrow V/_\sim$ (where $(r,\theta){\sim} (r',\theta')$ if $r=r'$ and $\theta=\theta'\pmod{2\pi}$).  Notice that $V/_\sim$ is a neighborhood of $\{0\}\times (\R/2\pi \Z)$ in $(-\delta,\delta)\times (\R/2\pi\Z)$. 
Likewise we define $\pi':V'\rightarrow V'/_\sim\subset (-\delta',\delta')\times (\R/2\pi\Z)$.

Since $h\circ T=T\circ h$, there exists a homeomorphism $\bar{h}:V/_\sim\rightarrow V'/_\sim$ such that 
\begin{equation}\label{E:diagCom}
\begin{array}[c]{ccc}
V&\stackrel{h}{\longrightarrow}&V'
\\
\downarrow\scriptstyle{\pi}&&\downarrow\scriptstyle{\pi'}
\\
V/_\sim & \stackrel{\bar{h}}{\longrightarrow} & V'/_\sim
\end{array}
\end{equation}
commutes.

Since $P$ and $P'$ do not have singularities on $(0,\delta)\times (\R/4\pi\Z)$ and $(0,\delta')\times(\R/4\pi\Z)$ respectively, 
then the line field spanned by each of them has no singularity and is $2\pi$-periodic with respect to the second variable.
Therefore, one can identify $\pi_*P$ and $\pi'_*P'$ with two line fields without singularities on $(0,\delta)\times(\R/2\pi\Z)$ and $(0,\delta')\times(\R/2\pi\Z)$ respectively.
The commutativity of diagram~\eqref{E:diagCom} and Lemma~\ref{L:TopEqLift} then show that $\bar{h}$ is a homeomorphism between $V/_\sim$ and $V'/_\sim$ that maps the integral manifolds of $\pi_*P$ onto the integral manifolds of $\pi'_*P'$.

Since $\psi\circ T =\psi$ on $(0,\delta)\times (\R/4\pi\Z)$, there exists a diffeomorphism $\widehat{\psi}:V_+/_\sim \rightarrow D_p(\delta)\backslash\{p\}$ such that $\widehat{\psi}\circ\pi=\psi$, where $V_+=V\cap ((0,\delta)\times (\R/4\pi\Z))$. In particular $\widehat{\psi}(V_+/_\sim)\cup \{p\}$ is a neighborhood of $p$.
Lemma~\ref{L:TopEqPolar} implies that $\widehat{\psi}^{-1}$ is a homeomorphism from $\widehat{\psi}(V_+/_\sim)$ onto $V_+/_\sim$ that maps the integral manifolds of $L$ onto the integral manifolds of $\pi_*P$. One can similarly define $\widehat{\psi}'$ on $V_+'/_\sim$, which satisfies  analogous properties.

The topological equivalence of $L$ and $L'$ at $p$ and $p'$ can
therefore be proven through the homeomorphism $H:\widehat{\psi}(V_+/_\sim)\cup\{p\}\to \widehat{\psi}'(V'_+/_\sim)\cup\{p\}$ defined by 
$H(p)=p'$ and 
$$
H:\widehat{\psi}(V_+/_\sim)\stackrel{\widehat{\psi}^{-1}}{\longrightarrow}
V_+/_\sim\stackrel{\bar{h}}{\longrightarrow}V_+'/_\sim
\stackrel{\widehat{\psi}'}{\longrightarrow}\widehat{\psi}'(V_+'/_\sim).
$$ 
By construction, $H$ is indeed a homeomorphism between a neighborhood of  $p$ and a neighborhood of $p'$ which takes the integral manifolds of $L$ onto those of $L'$.
\end{proof}

\section{The role of the metric $g$}
\label{Section:The_role_of_the_metric}

\subsection{What changes if we change the metric $g$}
\label{SubSection:What_changes_if_we_change_the_metric}

In this paper, the metric $g$ is fixed from the beginning and the main results (\MT , Propositions  \ref{P:GlobalDef} and \ref{P:indices}) are independent of its choice. It is natural to ask which properties are affected by the choice of $g$. The following example shows that a proto-line-field having a Lemon singularity at a point $p$ for a certain metric can have a Monstar singularity for another metric. Notice however that the Star singularity cannot be transformed into a Lemon or Monstar singularity by changing $g$, since they have different indices.

\begin{ex}\label{R:bifurcation}
For every $\lambda>0$ consider the Riemannian metric  $g_\lambda=dx^2+\lambda^2 dy^2$ on $\R^2$.
Let $X(x,y)=\begin{pmatrix}
x\\y
\end{pmatrix}$ and $Y(x,y)=\begin{pmatrix}
1\\0
\end{pmatrix}$. Then the singularity of $L=(X,Y)$ at $(0,0)$ can be a Lemon singularity or a Monstar singularity depending on $\lambda$.
Indeed, let $\phi_X(\theta)=
\angle_{g_\lambda}
\left[
	\begin{pmatrix}
	1 \\ 0
	\end{pmatrix},X(\cos\theta,\sin\theta)
\right]$ and $\phi_L(\theta)=
\angle_{g_\lambda}
\left[
	\begin{pmatrix}
	1 \\ 0
	\end{pmatrix}, B_{g_\lambda}(X(\cos\theta,\sin\theta),Y(\cos\theta,\sin\theta))
\right]=\frac{1}{2}\phi_X(\theta)$. Notice that $\theta=0$ is a fixed point of $\phi_X$ and let us compute 	the derivative of $\phi_X$ at $0$.

Since 
$$
g_\lambda\left(\begin{pmatrix}
1\\0
\end{pmatrix},\begin{pmatrix}
1\\0
\end{pmatrix}\right)=1,
\quad
 g_\lambda\left(\begin{pmatrix}
0\\1
\end{pmatrix},\begin{pmatrix}
0\\1
\end{pmatrix}\right)=\lambda^2,
$$ 
one has 
\begin{align*}
\| X(\cos\theta,\sin\theta) \|_{g_\lambda} \cos \phi_X(\theta) =
\cos\theta,\quad 
\lambda \| X(\cos\theta,\sin\theta) \|_{g_\lambda} \sin \phi_X(\theta) = 
\sin\theta.
\end{align*}
Therefore, $\tan\phi_X(\theta)=\lambda\tan\theta$ and thus $\frac{d}{d\theta}\phi_X(0)=\lambda$ and $\frac{d}{d\theta}\phi_L(0)=\lambda/2$. 

If $0<\lambda<2$, then $\theta=0$ is the only fixed point of $\phi_L$ and it is repulsive. If $\lambda>2$, then $\theta=0$ is an attractive fixed point of $\phi_L$. (See Figure~\ref{I:exBifurcation}.)

\begin{figure}[ht!]
\begin{center}
        \subfloat[$\lambda=1$]{\includegraphics[width=0.28\textwidth]{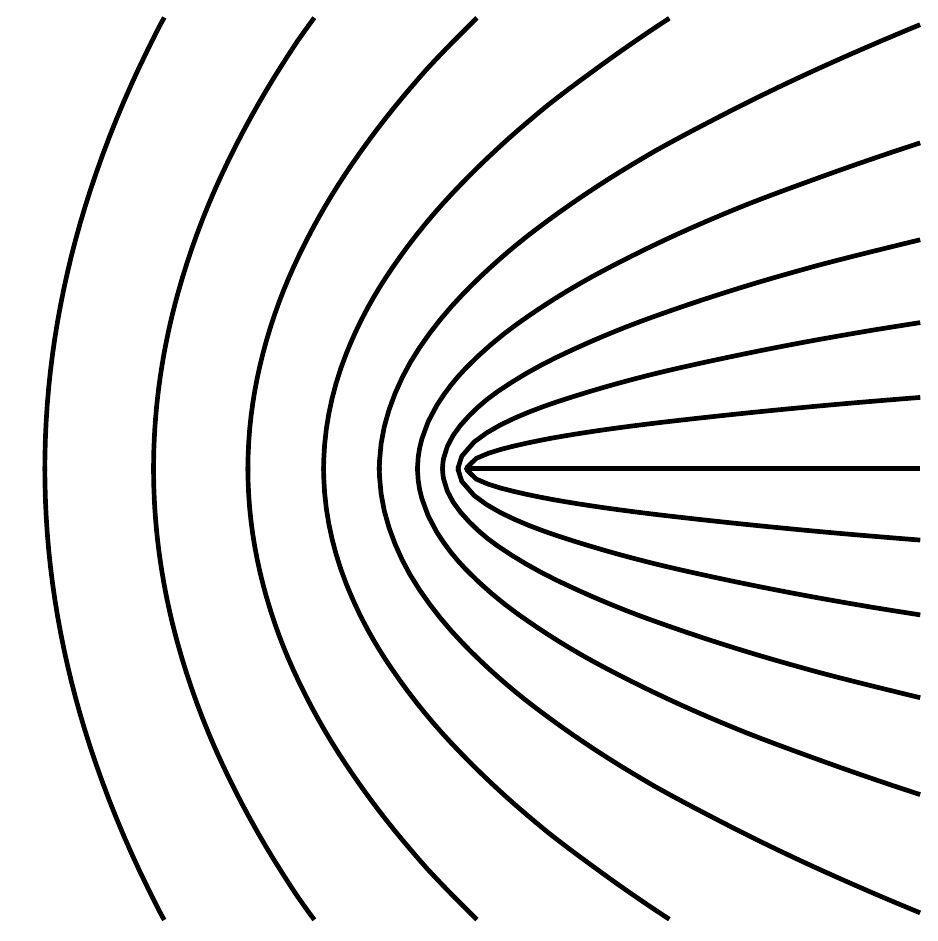}}
        \hspace{0.05\textwidth}
        \subfloat[$\lambda=2$]{\includegraphics[width=0.28\textwidth]{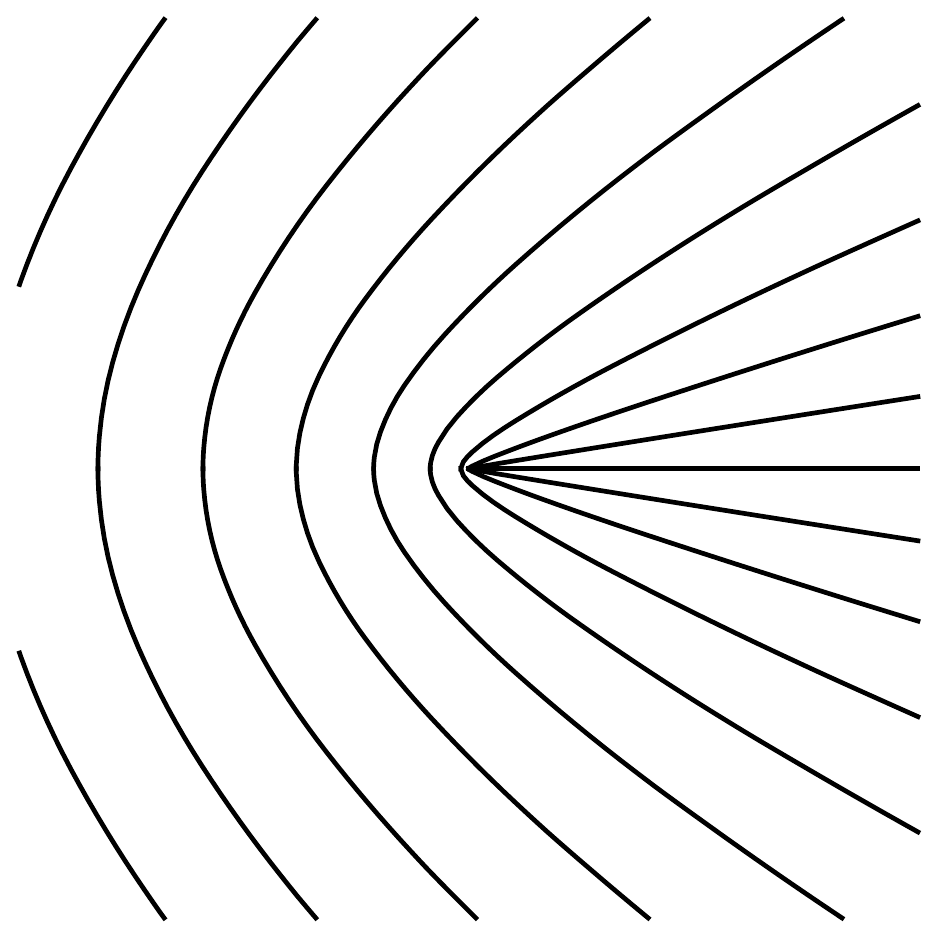}}
        \hspace{0.05\textwidth}
        \subfloat[$\lambda=3$]{\includegraphics[width=0.28\textwidth]{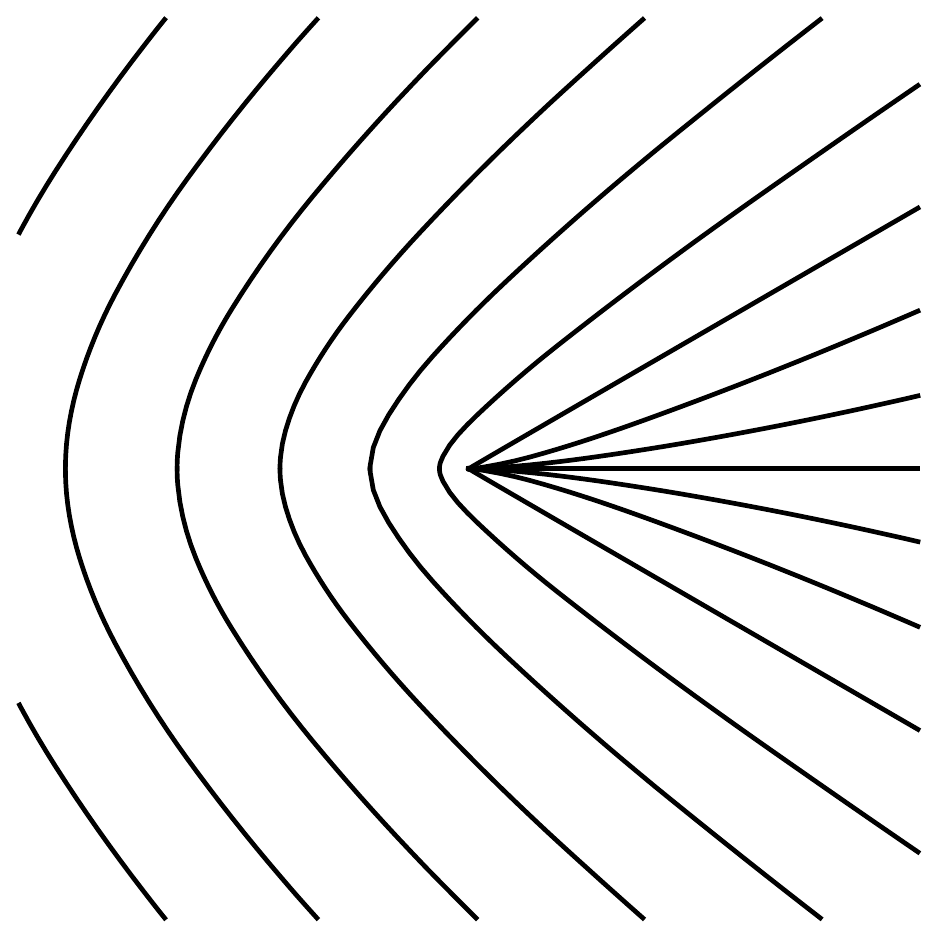}}
        \caption{Integral manifolds of the proto-line-fields of Example \ref{R:bifurcation} for three values of $\lambda$.}
        \label{I:exBifurcation}
\end{center}

\end{figure}

Notice that the bifurcation value $\lambda=2$ corresponds to a 
case which is not \superhyper. Hence a proto-line-field that has a structurally stable singularity at a point $p$ for a certain metric, can have non-structurally stable singularities at $p$  for another metric. 
\end{ex}

The next proposition shows that if we take a smooth curve $\gamma$ passing through a hyperbolic singularity of a proto-line-field $L$, then the angle between the line field associated with $L$ and $\dot{\gamma}$, measured with respect to the metric $g$, makes a jump of $\pi/2$ at the singularity. Hence, changing the metric and keeping the same proto-line-field, 
produces a new line field for which the angle between 
$\dot{\gamma}$ 
and itself, measured with respect to the new metric,
jumps again of $\pi/2$.

\begin{prop}\label{p:jump}
Let $L$ be a proto-line-field on $(M,g)$ with a hyperbolic singularity at $p\in M$. Let $\gamma:(-1,1)\rightarrow M$ be a smooth curve on $M$ such that $\gamma(0)=p$ and $\dot\gamma(0)\ne 0$. Then 
$$
\lim_{t\rightarrow 0^+} \angle_g\left[\dot{\gamma}(t), B(L(\gamma(t))) \right]
=
\lim_{t\rightarrow 0^-} \angle_g\left[ \dot{\gamma}(t), B(L(\gamma(t)))  \right]+\pi/2\pmod{\pi}.
$$
\end{prop}
\begin{proof}
Up to a change of parametrization, we can assume that $\gamma$ is parametrized by arc length, so that we can fix a system of coordinates $(x,y)$ on a neighborhood of $p=(0,0)$ such that $g(0,0)=\mathrm{id}$ and $\gamma$ coincides with the curve $t\mapsto(t,0)$. 

Then using the notation of Section~\ref{SubSection:Blowup}, for $t>0$ we have $\angle_g\left[\dot{\gamma}(t), B(L(\gamma(t)) )\right]=\phi_L(t,0)$ and $\angle_g\left[\dot{\gamma}(-t), B(L(\gamma(-t))) \right]=\phi_L(t,\pi)$.

Since $X$ has a hyperbolic singularity at $p$, we have 
that 
$$
\lim_{t\rightarrow 0^+} \frac{X(t,0)}{\|X(t,0)\|}
=-
\lim_{t\rightarrow 0^+}\frac{X(-t,0)}{\|X(-t,0)\|}.
$$
By definition of $L$, we then have
$$
\lim_{t\rightarrow 0^+} \phi_L(t,0)
=
\lim_{t\rightarrow 0^+} \phi_L(t,\pi)+\pi/2 \pmod{\pi}.
$$
\end{proof}

\begin{rk}
We know from Proposition~\ref{P:GlobalDef} that for every 
closed set $K\subset M$ and 
every section $L$ of $PT(M\setminus K)$, for every Riemannian metric $g$ on $M$, there exists a proto-line-field $(X,Y)$ such that $B(X,Y)=L$  on $M\setminus K$. 

Proposition~\ref{p:jump} says that, even if $K$ is made of isolated points, one cannot expect in addition that the 
singularities of 
$X$ and $Y$ are hyperbolic, unless some compatibility condition between $L$ and $g$ is satisfied at each point of $K$. 

In particular, 
a line field associated with a proto-line-field with hyperbolic singularities for a certain Riemannian metric is not in general associated with any proto-line-field with hyperbolic singularities for a different Riemannian metric.
\end{rk}

\subsection{How to construct a Riemannian metric from a pair of vector fields}
\label{SubSection:How_to_construct_a_Riemannian_metric}

The procedure of defining a line field by using two vector fields and a Riemannian metric (or a conformal structure) may look greedy and one may wonder if some alternative definition involving less functional parameters can lead to similar characterizations of structurally stable singularities. In this section we propose a way to get rid of the requirement of fixing a Riemannian metric. This is done by constructing 
a Riemannian metric from  two vector fields alone, at least in the generic case.

\begin{prop}\label{P:RiemMetricPair}
Let $X,Y$ be two vector fields on $M$ and set
\[
Z=\left[X,Y\right],\;W^1=\left[X,\left[X,Y\right]\right],\;W^2=\left[Y,\left[X,Y\right]\right].
\]
Generically with respect to $(X,Y)\in W^2(M)$, the vectors  $X(p),Y(p),Z(p),W^1(p),W^2(p)$ span $T_p M$ for every $p\in M$.
In this case 
\begin{equation}\label{E:genericpair1}
\left\|V\right\|=\min\left\{
\left\| u\right\|_2 \; | \; u\in \mathbb{R}^5,
\left( u_1 X+u_2 Y+u_3 Z+u_4 W^1+u_5 W^2 \right)(p)=V
\right\},\quad V\in T_pM,
\end{equation}
is a norm on $T_p M$ depending smoothly on $p$ and 
$$
g_{X,Y}(V,V')=\frac{1}{2}\left( \left\|V+V'\right\|^2 -\left\|V\right\|^2 -\left\|V'\right\|^2\right),\quad V,V'\in T_p M,
$$
defines a Riemannian metric  on $M$. 
\end{prop}
\begin{proof}[Proof of Proposition~\ref{P:RiemMetricPair}]
Up to reducing $M$ to a coordinate chart, the set
\[
\mathcal{F}=
\left\{j_2 (X,Y)(p)
\left|\begin{array}{l}
p\in M,\,X,Y\mbox{ vector fields on }M\\
X(p)\wedge Y(p)= 0\\
{[X,Y]}(p)\wedge X (p)=0\\
\left[X,\left[X,Y\right]\right](p)\wedge X(p)=0\\
\end{array}\right.\right\}
\]
can be identified with a submanifold of codimension 3 of $J^2\left(M,\mathbb{R}^4\right)$, the set of 2-jets of maps from $M$ to $\mathbb{R}^4$. 
By Thom's transversality theorem, for
a generic pair $(X,Y)\in W^2(M)$, 
there exist no $p\in M$ such that $j_2 (X,Y)(p)\in \mathcal{F}$. 
In particular, $X(p),Y(p),Z(p),W^1(p),W^2(p)$ span $T_p M$ for every $p\in M$.

Fix now $p\in M$, $V\in T_p M$ and let us give an explicit expression of the vector $u$ realizing the minimum in \eqref{E:genericpair1}, assuming that $X(p),Y(p),Z(p),W^1(p),W^2(p)$ span $T_p M$.
Write in local coordinates $V=(V_1,V_2)$, $X=(X_1,X_2)$, $Y=(Y_1,Y_2)$, {\it etc}. Let 
\begin{equation*}\label{E:definition_si}
s_i=\left(X_i,Y_i,Z_i,W^1_i,W^2_i\right)(p) \in \mathbb{R}^5, \; i\in \{1,2\},
\end{equation*}
and $r^4=|s_1|^2|s_2|^2-\left\langle s_1,s_2\right\rangle^2$, which is positive because $s_1$ and $s_2$ cannot be colinear.
Since the two affine hyperplanes  $H_1=\{u\in \mathbb{R}^5\mid \left\langle u,s_1 \right\rangle=V_1\}$ and $H_2=\{u\in \mathbb{R}^5\mid \left\langle u,s_2 \right\rangle=V_2\}$ are not parallel,  minimizing $\left\| u\right\|_2$ in \eqref{E:genericpair1} comes down to finding the orthogonal projection $u_V$ of $0_{\mathbb{R}^5}$ onto $H_1\cap H_2$.
Since the orthogonal subspace to $H_1\cap H_2$ is the span of $s_1$ and $s_2$, then the point $u_V$ is characterized by the conditions
\[
\left\langle u_V,s_1 \right\rangle =V_1,\quad \left\langle u_V,s_2\right\rangle =V_2,\quad u_V=\lambda_1 s_1+\lambda_2 s_2, \; \lambda_1,\lambda_2\in \mathbb{R}.
\]
Hence  $(\lambda_1,\lambda_2)$ is the unique solution of the system
\begin{equation}\label{syst1}
\left\{\begin{array}{ccc}
\lambda_1 |s_1|^2+\lambda_2 \left\langle s_1,s_2 \right\rangle &=&V_1\\
\lambda_1 \left\langle s_1,s_2 \right\rangle +\lambda_2 |s_2|^2&=&V_2,
\end{array}\right.
\end{equation}
which leads to the characterization of $u_V$ as 
\begin{equation*}\label{eq:u_V}
u_V=\left(\frac{V_1 |s_2|^2 -V_2 \left\langle s_1,s_2\right\rangle}{r^4} \right) s_1 + \left(\frac{V_2 |s_1|^2 -V_1 \left\langle s_1,s_2\right\rangle}{r^4}  \right) s_2.
\end{equation*}

Since $u_V$ depends linearly on $V$, it is then easy to see that the norm  $\left\|\cdot \right\|=\|u_\cdot\|_2$ depends smoothly on $p$. This norm derives from a scalar product since it verifies the parallelogram law, as we are now going to show. 
Let $V,V'\in T_pM$ and $u_V$, $u_{V'}$ be defined as above. 
Then $u_V=\lambda_1 s_1+ \lambda _2 s_2$ and $u_{V'}=\lambda'_1 s_1+ \lambda' _2 s_2$, where $(\lambda_1,\lambda _2)$ and $(\lambda'_1 ,\lambda' _2)$ are the respective unique solutions of  system \eqref{syst1}. Then, by linearity, we have that $u_{V+V'}=(\lambda_1+\lambda'_1) s_1+ (\lambda_2+\lambda'_2) s_2$ and $u_{V-V'}=(\lambda_1-\lambda'_1) s_1+ (\lambda_2-\lambda'_2) s_2$.
Thus 
\begin{align*}
\left\|V\right\|^2+\left\|V'\right\|^2&=\left\|u_V\right\|_2^2+\left\|u_{V'}\right\|_2^2
=\left\|u_V+u_{V'}\right\|_2^2+\left\|u_V-u_{V'}\right\|_2^2
=\left\|V+V'\right\|^2+\left\|V-V'\right\|^2.
\end{align*}
\end{proof}

\begin{rk}
Notice that for every compact $K\subset M$ the set of pairs $(X,Y)$ such that 
the metric $g_{X,Y}$ introduced in Proposition~\ref{P:RiemMetricPair} is well-defined on $K$ 
is open in $\W^2(M)$ and $g_{X,Y}$ depends continuously on $(X,Y)$ on it. 
Hence, because of the continuity of the linearization of a proto-line-field with respect to $(X,Y)$ and $g$ (see Definition~\ref{d:linearization}), we deduce the local structural stability of Lemon, Monstar and Star singularities for proto-line-fields $(X,Y)$ with respect to the metric $g_{X,Y}$, in the sense that if at a singular point $p$ the linearized system is satisfies condition 1, 2, or 3 of Theorem \ref{T:thmcas}, 
then the same property is satisfied by every small perturbation of $(X,Y)$ in the $\W^2(M)$ topology.

In order to prove 
that generically with respect to $(X,Y)$
the singularities of the proto-line-field $(X,Y)$ with respect to the metric $g_{X,Y}$
are Darbouxian, one should further prove that non-Darbouxian singularities can be removed by small perturbations of $(X,Y)$. 
Although we expect this result to be true (in the $\W^2(M)$ topology and not in the $\W^1(M)$ one as it was the case in Theorem~\ref{T:MainTheorem}), this does not follow directly from the results in this paper.
\end{rk}

\appendix
\section{Proof of Proposition \ref{P:SigneDer}}
\label{Appendix:proof_of_Proposition}
Let us first compute $F'$ and $G'$. By definition of $F$ we have
$\sin\theta\cos F(\theta)=E\cos\theta\sin F(\theta)$ 
from which we obtain
\begin{align*}
E\cos^2\theta\sin F(\theta)-F'(\theta) \sin^2\theta\sin F(\theta) &=E^2  F'(\theta) \cos^2\theta\sin F(\theta)- E \sin^2\theta\sin F(\theta). \\
\shortintertext{Then either $\sin F(\theta)=0$, and then $\sin\theta=0$, or}
E\cos^2\theta-F'(\theta) \sin^2\theta &=E^2  F'(\theta) \cos^2\theta- E \sin^2\theta .
\end{align*}
By smoothness of $F$,
\begin{equation}\label{E:dF}
F'(\theta)=\frac{E}{E^2\cos^2\theta+\sin^2\theta},\quad \forall\theta\in [-\pi,\pi).
\end{equation}

Concerning $G$, in the case of focuses,
we immediately get that $G(\theta)=F(\theta+\varphi)$. In the other two cases, reasoning as for $F$, we get 
\begin{equation}\label{E:dG}
G'(\theta)=\frac{E\cos(2\varphi)}{E^2\cos^2(\theta-\varphi)+\sin^2(\theta+\varphi)},\quad\forall \theta\in [-\pi,\pi).
\end{equation}

We prove Proposition~\ref{P:SigneDer} by considering separately the three cases corresponding to the type of singularity of $X$. The saddle case follows immediately from \eqref{E:dF} and \eqref{E:dG}, since 
 $\varphi \in (\pi/4,3\pi/4)\pmod{\pi}$ implies that   $\cos(2\varphi)<0$.
The following two lemmas consider the focus and node case, respectively.

\begin{lem}
Let the singularity of $X$ be a focus and set $A=\sqrt{5-2\cos(2\varphi)}$. Then $\Phi$ and $\kappa$ as in the statement of Proposition \ref{P:SigneDer} exist and are characterized by 
\begin{equation}\label{eq:firstA}
\left\{\begin{array}{rl}
A\cos(\Phi)=&2\cos(2\varphi)-1\\
A\sin(\Phi)=&2\sin(2\varphi)
\end{array}\right.
\end{equation}
and \[\kappa=\frac{E^2+1}{A(E^2-1)}.\]
\end{lem}

\begin{proof}
Let $\Phi$ satisfy \eqref{eq:firstA}. 
By definition of $\Phi$ and $A$, the inequality 
\[
\cos(2\theta+\Phi)>-\frac{E^2+1}{A(E^2-1)}\]
is equivalent to
\begin{equation}\label{**}
(E^2-1)\cos(2\theta)(2\cos(2\varphi)-1)-2(E^2-1)\sin(2\theta)\sin(2\varphi)+E^2+1>0.
\end{equation}
By elementary trigonometric identities, condition \eqref{**} is equivalent to 
\[
2(E^2\cos^2(\theta+\varphi)+\sin^2(\theta+\varphi))- (E^2\cos^2\theta+\sin^2\theta )>0,
\]
which, in turns,  is equivalent to
\[
2F'(\theta)-G'(\theta)=\frac{2E}{E^2\cos^2\theta+\sin^2\theta}-\frac{E}{E^2\cos^2(\theta+\varphi)+\sin^2(\theta+\varphi)}>0.
\]
\end{proof}

\begin{lem}
Let the singularity of $X$ be a node and set $$A=\sqrt{\frac{1}{2} \left(5+6E^2+5E^4 -(3+10 E^2+3E^4)\cos4\varphi\right) }.$$ Then $\Phi$ and $\kappa$ as in the statement of Proposition \ref{P:SigneDer} exist and are characterized by 
\begin{equation}\label{eq:secondA}
\left\{\begin{array}{rl}
A\cos(\Phi)=&(E^2-1)\cos(2\varphi)\\
A\sin(\Phi)=&-2(E^2+1)\sin(2\varphi)
\end{array}\right.
\end{equation}
and \[\kappa=\frac{(E^2+1)}{A}\left(2-\cos2\varphi\right).\]
\end{lem}

\begin{proof}
In this case we have
\[
2F'(\theta)-G'(\theta)=\frac{2E}{E^2\cos^2\theta+\sin^2\theta}-\frac{E\cos(2\varphi)}{\sin^2(\theta+\varphi)+E^2\cos^2(\theta-\varphi)},\]
and it follows by elementary trigonometric identities that $2F'(\theta)-G'(\theta)>0$ if and only if 
\[
(E^2-1)\cos 2\varphi \cos 2\theta +2(E^2+1)\sin 2\varphi\sin 2\theta +(1-E^2)(2-\cos 2\varphi)>0.
\]
By definition of $A$ and letting $\Phi$ satisfy \eqref{eq:secondA}, this inequality is equivalent to
\[
A\cos(2\theta+\Phi)>-\left(2(E^2+1)+(1-E^2)\cos2\varphi\right).
\]
\end{proof}

\section{Extension of the direction at blown-up singularities}
\label{Appendix:Extension_of_the direction}
\begin{lem}\label{L:blowup}
Let $L=(X,Y)$ be a proto-line-field on $(M,g)$ with a hyperbolic singularity at $p\in M$. Fix a system of coordinates $(x,y)$ such that $p=(0,0)$, $g(0,0)=\mathrm{id}$.
Assume that $Y(0,0)\not =0$ and consider the linear  proto-line-field $\bar{L}=(DX(0,0),Y(0,0))$.

For every $r>0$ small enough and $\theta \in \R$, 
let $\phi_L(r,\theta)$ and  $\phi_{\bar{L}}(\theta)$ in $\mathbb{R}\slash \pi \mathbb{Z}$ be defined by
$$
\phi_L(r,\theta)=
\angle_{g}
\left[
	\begin{pmatrix}
	1 \\ 0
	\end{pmatrix},B(L)(r\cos\theta,r\sin\theta)
\right]
\text{ and } 
\phi_{\bar{L}}(\theta)=
\angle_{\mathrm{Eucl}}
\left[
	\begin{pmatrix}
	1 \\ 0
	\end{pmatrix},B(\bar{L})(r\cos\theta,r\sin\theta)
\right].
$$
Then
$$
\phi_L(r,\theta) \underset{r\rightarrow 0}{\longrightarrow} \phi_{\bar{L}}(\theta),
\quad
\frac{\partial\phi_L}{\partial\theta}(r,\theta) \underset{r\rightarrow 0}{\longrightarrow} \frac{d\phi_{\bar{L}}}{d\theta}(\theta),\text{ and }\quad
\frac{\partial\phi_L}{\partial r}(r,\theta) \underset{r\rightarrow 0}{\longrightarrow} 0.
$$
\end{lem}
\begin{proof}
Let $\bar{X}(x,y)=DX(0,0)	\begin{pmatrix}
	x \\ y
	\end{pmatrix}
  $ and define $\phi_X(r,\theta)$, $\phi_Y(r,\theta)$,  ${\phi}_{\bar X}(\theta)$, and $\alpha$ in $\mathbb{R}\slash 2\pi \mathbb{Z}$ by
\begin{align*}
\phi_X(r,\theta)&=
\angle_{g}
\left[
	\begin{pmatrix}
	1 \\ 0
	\end{pmatrix},X(r\cos\theta,r\sin\theta)
\right],&\quad 
{\phi}_{\bar X}(\theta)&=
\angle_{\mathrm{Eucl}}
\left[
	\begin{pmatrix}
	1 \\ 0
	\end{pmatrix},\bar X(\cos\theta,\sin\theta)
\right],\\
\phi_Y(r,\theta)&=
\angle_{g}
\left[
	\begin{pmatrix}
	1 \\ 0
	\end{pmatrix},Y(r\cos\theta,r\sin\theta)
\right],&
\alpha &=
\angle_{\mathrm{Eucl}}
\left[
	\begin{pmatrix}
	1 \\ 0
	\end{pmatrix},Y(0,0)	
\right].
\end{align*}

Since $\phi_L=\frac{1}{2}\phi_X+\frac{1}{2}\phi_Y\pmod{\pi}$ and $\phi_{\bar{L}}=\frac{1}{2}\bar{\phi}_X+\frac{1}{2}\bar{\phi}_Y\pmod{\pi}$, we are left to prove that 
$$
\phi_X(r,\theta) \underset{r\rightarrow 0}{\longrightarrow} {\phi}_{\bar X}(\theta), \quad
\frac{\partial \phi_X}{\partial\theta}(r,\theta) \underset{r\rightarrow 0}{\longrightarrow} \frac{d \phi_{\bar X}}{d\theta}(\theta),
\quad
\frac{\partial \phi_X}{\partial r}(r,\theta)\underset{r\rightarrow 0}{\longrightarrow} 0,
$$
and 
$$
\phi_Y(r,\theta) \underset{r\rightarrow 0}{\longrightarrow} \alpha, \quad
\frac{\partial \phi_Y}{\partial\theta}(r,\theta) \underset{r\rightarrow 0}{\longrightarrow} 0,
\quad
\frac{\partial \phi_Y}{\partial r}(r,\theta)\underset{r\rightarrow 0}{\longrightarrow} 0.
$$

We are going to give a proof for $\phi_X$ only, the one for $\phi_Y$ being analogous. 

Using the local coordinates $(x,y)$, let us identify vector fields 
with their coordinate representation. 
Denote by $(e_1,e_2)$ an oriented orthonormal frame for $g$ in a neighborhood of $p$ such that $e_1$ is a positive multiple of $\begin{pmatrix}
	1 \\ 0
	\end{pmatrix}$.
Since $g(0,0)=\mathrm{id}$, then $e_1(0,0)=\begin{pmatrix}
	1 \\ 0
	\end{pmatrix}=\bar e_1$ and $e_2(0,0)=\begin{pmatrix}
	0 \\ 1
	\end{pmatrix}=\bar e_2$. 
For $i=1,2$,
\begin{align*}
g(e_i,X)-\langle \bar e_i,\bar X\rangle
&=
g(e_i,X)-\langle e_i,X \rangle
+
\langle e_i-\bar{e}_i,X\rangle 
+
\langle \bar{e}_i,X-\bar{X}\rangle.
\end{align*}

By an abuse of notation, we write in what follows $X(r,\theta)$ 
for  $X(r\cos\theta,r\sin\theta)$ and 
similarly for $e_1$, $e_2$ and $\bar X$.
By definition of $\bar{X}$, we have $\left\| (X - \bar{X})(r,\theta) \right\|_{2}=O(r^2)$, so that  $\langle \bar{e}_i,X-\bar{X}\rangle =O(r^2)$.
Likewise $\left\| e_i(r,\theta) - \bar{e}_i \right\|_{2}=O(r)$ and $\left\| X(r,\theta) \right\|_{2}=O(r)$, so that $\langle e_i-\bar{e}_i,X\rangle =O(r^2)$. Finally $| g(W,Z)-\langle W,Z\rangle |\le \varepsilon(r)\|W\|_2 \|Z\|_2 $, with $\varepsilon(r)=O(r)$, for any pair of vector fields $W,Z$, so that $g(e_i,X)-\langle e_i,X\rangle =O(r^2)$.

In conclusion, $g(e_i,X)-\langle \bar{e}_i,\bar{X}\rangle =O(r^2)$. 
By definition  of $\phi_X$, we have
\begin{equation}\label{E:defTh}
\cos (\phi_X) g(e_2, X) = \sin (\phi_X)  g(e_1, X), 
\end{equation}
which can be rewritten as 
\begin{multline}\label{E:difLin}
\cos (\phi_X)\langle \bar{e}_2,\bar{X}\rangle +\cos (\phi_X)( g(e_2, X)-\langle \bar{e}_2,\bar{X}\rangle) \\= \sin (\phi_X)\langle \bar{e}_1,\bar{X}\rangle + \sin (\phi_X) (g(e_1, X) -\langle \bar{e}_1,\bar{X}\rangle).
\end{multline}
Dividing  equation \eqref{E:difLin} by $r$,
we get
\begin{align}\label{eq:o0}
\phi_X(r,\theta)&=\phi_{\bar{X}}(\theta)+O(r).
\end{align}

Regarding the partial derivatives of $\phi_X$, by differentiating \eqref{E:defTh} we get
\begin{align}
\partial_\theta \phi_X
\left(
	\cos \phi_X
	g(e_1, X)
	+
	\sin \phi_X
	g(e_2, X)
\right)
& =
\cos  \phi_X
\partial_\theta
(g(e_2, X))
-
\sin  \phi_X
\partial_\theta
(g(e_1, X)).\label{eq:derphi}
\end{align}

We have
\[
\partial_\theta (g(e_i, X))
=
(\partial_\theta g)(e_i, X)
+
g(\partial_\theta e_i, X)
+
g(e_i, \partial_\theta X).
\]
By singularity of the polar parameterization, 
we have $\partial_\theta g=O(r)$ and $\partial_\theta e_i=O(r)$. 
Moreover, $g(e_i, \partial_\theta X)(r,\theta)=\langle \bar{e}_i, \partial_\theta \bar{X}(r,\theta)\rangle +O(r^2)=r\langle \bar{e}_i, \partial_\theta \bar{X}(1,\theta)\rangle +O(r^2)$.

By definition of $\phi_X$, $\cos \phi_X g(e_1, X)+\sin \phi_X g(e_2, X)=\left\| X(r,\theta) \right\|_g=r\left\|r\bar X(1,\theta) \right\|_2+r $.
Hence we deduce from \eqref{eq:o0} and \eqref{eq:derphi} that
\begin{align*}
\frac{
\partial
	\phi_X
}{
	\partial \theta 
}
(r,\theta)
&=
\frac{
	\cos {\phi}_{\bar X}(\theta)
	\langle \bar{e}_2, \partial_\theta \bar{X}(1,\theta)\rangle
	-
	\sin {\phi}_{\bar X}(\theta)
	\langle \bar{e}_2, \partial_\theta \bar{X}(1,\theta)\rangle 
	+O(r)
}{
	\left\| 
	\bar X(1,\theta) 
	\right\|_2+O(r)
    }
 =
\frac{
	d \bar{\phi}_{X}
}{
	d\theta 
}
(\theta)+O(r).
\end{align*}

Similarly, we have 
\[
\partial_r \phi_X
\left\|
	X
\right\|_g
=
\cos  \phi_X
\partial_r
(g(e_2, X))
-
\sin  \phi_X
\partial_r
(g(e_1, X)).
\]

Since
\begin{align*}
\partial_r (g(e_i, X))(r,\theta)=r\langle \bar{e}_i, \bar{X}(1,\theta)\rangle +O(r^2)
\end{align*}
and
\[
\cos {\phi}_{\bar X}\langle \bar{e}_2, \bar{X}(1,\theta)\rangle -\sin {\phi}_{\bar X}\langle\bar{e}_1, \bar{X}(1,\theta)\rangle =0,
\]
we have
\begin{align*}
\frac{\partial \phi_X}{\partial r} 
=
\frac{
	\cos  \phi_X
	\partial_r
	(g(e_2, X))
	-
	\sin  \phi_X
	\partial_r
	(g(e_1, X))
}{
	\left\| 
	X
	\right\|_g
}
= O(r).
\end{align*}
\end{proof}

\bibliographystyle{plain}
\bibliography{biblio}

\end{document}